\documentclass[12pt]{amsart}

\usepackage[top=1.6cm,bottom=2cm,left=1.1in,right=1.1in]{geometry}
\usepackage[T1]{fontenc}
\usepackage[utf8]{inputenc}
\usepackage{lmodern}
\usepackage{amsmath,amssymb,amsthm,mathtools}
\usepackage{mathrsfs}
\usepackage{enumitem}
\usepackage{microtype}
\usepackage[colorlinks=true,linkcolor=blue,citecolor=blue,urlcolor=blue]{hyperref}

\numberwithin{equation}{section}

\newtheorem{theorem}{Theorem}[section]
\newtheorem{lemma}[theorem]{Lemma}
\newtheorem{corollary}[theorem]{Corollary}
\newtheorem{proposition}[theorem]{Proposition}

\theoremstyle{definition}
\theoremstyle{remark}
\newtheorem{remark}[theorem]{Remark}
\theoremstyle{definition}
\newtheorem{definition}[theorem]{Definition}

\newcommand{\BC}{\mathrm{BC}}

\newcommand{\MK}{\mathcal{MK}}
\newcommand{\MN}{\mathcal{MN}}

\title[Boundary Cases of the \(J\)-Equation]
{Boundary Cases of the $J$-Equation:
Divisorial Rigidity and a Global $C^0$ Estimate}

\author{Jixiang Fu}
\address{Shanghai Center for Mathematical Sciences,
  Fudan University,
  Shanghai 200433, China}
\email{majxfu@fudan.edu.cn}

\author{Ziyi Zhang}
\address{School of Mathematical Sciences, Fudan University,
Shanghai 200433, China}
\email{21210180101@m.fudan.edu.cn}

\date{}

\begin{document}

\begin{abstract}
We study two  boundary cases of the stability condition for the
$J$-equation. First, under the $J$-semistable condition, we show that the
destabilizing prime divisors form an exceptional family in the sense of Boucksom, with a uniform numerical
gap away from them. The related modified nef estimate can remove the $J$-big assumption in Liu's work~\cite{Liu2026Boundary}.
 Second, under the smooth boundary cone condition, we obtain a uniform global \(C^0\) estimate
for solutions of the approximating twisted \(J\)-equations. As a consequence, we construct a bounded-potential Bedford--Taylor solution
of the \(J\)-equation, which is smooth outside the destabilizing prime
divisors.

\end{abstract}

\maketitle
\section{Introduction}

Let $(M^n,\omega)$ be a compact K\"ahler manifold of complex
 dimension $n\geq 2$. Let $\chi$ be another K\"ahler form.  The $J$-equation,
introduced by Donaldson \cite{Donaldson1999} and Chen \cite{Chen2000},
seeks a smooth function $u$ solving
\[
    \chi_u^n
    =
    n\chi_u^{n-1}\wedge\omega,
    \qquad
    \chi_u:=\chi+\sqrt{-1}\partial\bar\partial u>0,
\]
under the normalization
\(
    \displaystyle\int_M\chi^n
    =
    n\displaystyle\int_M\chi^{n-1}\wedge\omega.
\) The $J$-equation is closely related to the $J$-flow and to the geometry of
the space of K\"ahler metrics. 

Song--Weinkove \cite{SongWeinkove} established the existence of a smooth solution under the smooth subsolution condition (or cone condition). Lejmi--Sz\'ekelyhidi \cite{LejmiSzekelyhidi2015} subsequently proposed a numerical conjecture for solvability, which was proved by Collins--Sz\'ekelyhidi \cite{CollinsSzekelyhidi2017} on toric manifolds. A breakthrough was achieved by Chen \cite{Chen}, who established the uniform $J$-stable criterion and developed a powerful analytic framework for the $J$-equation. Building on these developments, Datar--Pingali \cite{DatarPingali} obtained a general numerical criterion in the projective case, and Song \cite{Song} finally proved the $J$-stable criterion on K\"ahler manifolds.

\begin{theorem}[$J$-stable criterion, see Theorem 1.1 of Song~\cite{Song}]\label{J-stable criterion}
    Let $(M^n,\omega)$ be a compact K\"ahler manifold of complex
 dimension $n\geq 2$. Let $\chi$ be another K\"ahler form. Suppose that the pair $(\chi,\omega)$ is $J$-stable, that is,
  $\displaystyle\int_M \left(\chi^n-n\chi^{n-1}\wedge\omega\right)\geq 0,$ and for any $p$-dimensional subvarieties $V\subset M$ with $1\leq p\leq n-1$, \(
    \displaystyle\int_V \left(\chi^p-p\chi^{p-1}\wedge\omega\right)>0.\)
    Then there exists a smooth function $u$ such that  $\chi_u$ is a K\"ahler form and satisfies the cone condition $$\chi_u^{n-1}-(n-1)\chi_u^{n-2}\wedge\omega>0.$$
\end{theorem} 

It is natural to ask what canonical solutions one
should expect for the $J$-equation when the stability condition fails, as suggested in~\cite{SongWeinkove, Song,  DatarMeteSong2026}.
There has been substantial progress on this question; see, for instance,
\cite{ChenXu2026, DatarMeteSong2026, FangLai12, FangLaiSongWeinkove2014, Fu2026, KD24, KhalidDyrefelt2026, Liu2026Semistable, Liu2026Boundary, Murakami2026,SongWeinkove, Sun2024Boundary, To2023}.
The purpose of this paper is to study two closely related boundary
cases in which the $J$-stable condition may fail, with particular emphasis on the role of destabilizing prime divisors in the degeneration of the $J$-equation.

The first case is the  $J$-semistable condition:
  \begin{equation}\label{topintegralnonnegative}
      \int_M \left(\chi^n-n\chi^{n-1}\wedge\omega\right)\geq 0,
      \end{equation} and for any $p$-dimensional subvarieties $V\subset M$ with $1\leq p\leq n-1$,
 \begin{equation}\label{$J$-semistablecondition}
    \int_V \left(\chi^p-p\chi^{p-1}\wedge\omega\right)\geq 0.\end{equation}
   The second case is the smooth boundary cone condition.  In this case, we assume the normalization \begin{equation}\label{normalizationcondition}\int_M\chi^n=n\int_M\chi^{n-1}\wedge \omega,
\end{equation} and
    \begin{equation}\label{semisubsolutioncondition}
        \chi^{n-1}-(n-1)\chi^{n-2}\wedge\omega\geq0.
        \end{equation}
 Clearly, the smooth boundary cone condition implies the $J$-semistable condition.
 
In order to identify the possible
locus along which the $J$-equation may degenerate, it is
natural to consider the destabilizing set introduced by Khalid
--Dyrefelt~\cite{KhalidDyrefelt2026}.  For $1\leq p \leq n$, we denote the $(p,p)$ form \begin{equation}\label{P_m(chi)}
F_p(\chi):=\chi^{p}-p\chi^{p-1}\wedge\omega.
\end{equation}
 An irreducible $p$-dimensional subvariety $Z$ satisfying $\int_ZF_{p}(\chi)=0$ will be called a destabilizing $p$-dimensional subvariety.
 
 Fang--Lai~\cite{FangLai12} and Datar--Mete-Song~\cite{DatarMeteSong2026} studied many examples on Calabi ansatz, providing important insights on the global $C^0$ estimate of the $J$-flow and its regularity outside the destabilizing subvarieties. Fang--Lai--Song--Weinkove~\cite{FangLaiSongWeinkove2014} proved that on K\"ahler surfaces, under the smooth boundary cone condition, the $J$-flow converges to a continuous potential solution, with smooth convergence outside the destabilizing curves. Recently, Liu~\cite{Liu2026Semistable} proved the high-order regularity outside the destabilizing subvarieties on K\"ahler threefolds, and later Liu~\cite{Liu2026Boundary} generalized the results to higher dimensions under the  $J$-big assumption.

From a different perspective, Khalid--Dyrefelt~\cite{KhalidDyrefelt2026} studied the structure of the destabilizing subvarieties, and  proved, under some positivity assumptions on the factor classes $[\chi-m\omega]$,  the finiteness of irreducible destabilizing analytic subvarieties of dimension $m$. X. Fu~\cite{Fu2026} proved an analogous result on toric manifolds.  Recently, Liu~\cite{Liu2026Boundary} suggested that the finiteness of destabilizing prime divisors can be obtained by a $J$-big assumption.

In this paper,  we focus on the destabilizing prime divisors. Our first result is a  rigidity property: the destabilizing prime divisors form an exceptional family in the sense of Boucksom \cite{Boucksom}.  More generally, we obtain a quantitative numerical gap defined as in equation~\eqref{theboundc}.
   \begin{theorem}
    \label{t1}
 Let $(M^n,\omega)$ be a compact K\"ahler manifold of complex
 dimension $n\geq 2$. Let $\chi$ be another K\"ahler form.  Suppose that the $J$-semistable condition ~\eqref{topintegralnonnegative}--~\eqref{$J$-semistablecondition} holds. Set the positive constant \begin{equation}\label{theboundc}
     c= \frac{2n\displaystyle\int_M\Big(\chi^n-n\chi^{n-1}\wedge\omega\Big)+\int_M \omega^n}{8n^2\displaystyle\int_M\chi\wedge\omega^{n-1}},
     \end{equation}and let \begin{equation}\label{S}
     S:=\{ D \,|\,D\; \text{is a prime divisor and } \int_DF_{n-1}(\chi)< c\int_D\omega^{n-1} \}.
 \end{equation}
Then any finite subset of  $S$ is an exceptional family in the sense of Boucksom~\cite{Boucksom}. In particular, the cardinality of  $S$ is bounded by  the Picard number of $M$.
\end{theorem}

Theorem~\ref{t1} has the geometric meaning that under the $J$-semistable condition, the destabilizing prime divisors are finite, and have negative divisorial directions in the sense of Boucksom. Theorem~\ref{t1} is proved by our following estimate.

\begin{proposition}\label{keyestimate-intro}
  Let $(M^n,\omega)$ be a compact K\"ahler manifold of complex
 dimension $n\geq 2$. Let $\chi$ be another K\"ahler form. Suppose that the $J$-semistable condition~\eqref{topintegralnonnegative}--~\eqref{$J$-semistablecondition} holds.
Then for any modified nef class $\beta$, we have \begin{equation}\label{ineq:keyestimate-intro}
    \int_M F_{n-1}(\chi)\wedge\beta \geq c\int_M \omega^{n-1}\wedge \beta,
    \end{equation}
    where \begin{equation}\label{intro-c=}
   c= \frac{2n\displaystyle\int_M\Big(\chi^n-n\chi^{n-1}\wedge\omega\Big)+\int_M \omega^n}{8n^2\displaystyle\int_M\chi\wedge\omega^{n-1}}.\end{equation}
\end{proposition}

We note that the modified nef estimate~\eqref{ineq:keyestimate-intro} in Proposition~\ref{keyestimate-intro} verifies the $J$-bigness criterion in~\cite[Proposition 29]{Liu2026Boundary}. Consequently, the
$J$-big assumption in Liu's~\cite{Liu2026Boundary} work  can be removed. 

While this paper was in preparation, under the  $J$-big assumption and the smooth boundary cone condition, Liu~\cite{Liu2026Boundary} independently established several results including
the negative definiteness  of an intersection form  and the construction of a singular barrier.

Since Proposition~\ref{keyestimate-intro} can remove the $J$-big assumption in ~\cite{Liu2026Boundary}, the results in ~\cite{Liu2026Boundary} hold under the smooth boundary cone condition. In particular,  there exists a
singular barrier with logarithmic singularities along the destabilizing prime divisors, as expected by Song--Weinkove~\cite[Remark 4.6]{SongWeinkove}. 
Using this singular barrier, we prove the global \(C^0\) estimate for the approximating twisted $J$-equations, and consequently we obtain a bounded Bedford--Taylor
solution. 
  
\begin{theorem}\label{thm:regularity of the weak solution}
  Let $(M^n,\omega)$ be a compact K\"ahler manifold of complex
 dimension $n\geq 3$. Let $\chi$ be another K\"ahler form. Suppose the normalization ~\eqref{normalizationcondition} and the smooth boundary cone condition~\eqref{semisubsolutioncondition} hold. Let \(\{D_1,\ldots,D_m\}\) be the set of all prime divisors satisfying
\begin{equation}
 \int_{D_i} F_{n-1}(\chi)=0.
 \end{equation}
Let \(Z:=\bigcup_{i=1}^m D_i\) be the analytic subvariety of codimension 1. 
Then there exists a function $u_{\infty}\in L^{\infty}(M)\cap C^{\infty}(M\setminus Z)$, with $\text{sup}_M u_{\infty}=0$,  satisfying 
\(\chi+\sqrt{-1}\partial\bar\partial u_\infty
\geq \omega\)
in the sense of currents, and
\begin{equation}\label{eq:Section-limit-equation}
(\chi+\sqrt{-1}\partial\bar\partial u_\infty
)^n=n(\chi+\sqrt{-1}\partial\bar\partial u_\infty
)^{n-1}\wedge\omega
\end{equation}
on \(M\) in the Bedford--Taylor sense. In particular, the equation
holds smoothly on \(M\setminus Z\).
\end{theorem}

To the best of our knowledge, the global $C^0$ estimate is new in this setting (even under the assumption of the existence of such a singular barrier; see~\cite{SongWeinkove}). As a further consequence, we also prove a uniform
\(C^0\) estimate for the \(J\)-flow under the smooth boundary cone condition, which confirms the
\(C^0\) estimate expectation by Song--Weinkove~\cite[Remark~4.4]{SongWeinkove}.

The paper is organized as follows. In Section 2 we recall the modified Kähler and modified
nef cones and Boucksom's notion of an exceptional family. In Section 3 we prove Proposition~\ref{keyestimate-intro} and then derive Theorem~\ref{t1} from it. In Section 4 we
study the smooth boundary cone case. Using the  singular barrier, we establish
local and global $C^0$ estimates for the approximating twisted J-equations, obtain local
higher-order estimates outside the destabilizing divisors, and pass to the limit to prove Theorem~\ref{thm:regularity of the weak solution}. We conclude with the uniform $C^0$ estimate for the
$J$-flow.

\section{Preliminaries}
In this section, for the reader's convenience, we briefly recall  Boucksom's~\cite{Boucksom} definition of exceptional family of prime divisors. Fix a Hermitian form $\omega_0$ on a compact manifold $X$.

\begin{definition}\label{def:modified-cones}
A class $\alpha\in H^{1,1}_{\BC}(X,\mathbb{R})$ is \emph{modified K\"ahler} if it contains a K\"ahler current $T$ whose generic Lelong number along any prime divisor is zero. It is \emph{modified nef} if, for any $\varepsilon>0$, it contains a closed current $T_\varepsilon\geq-\varepsilon\omega_0$ whose generic Lelong number along any prime divisor is zero.
\end{definition}
The set of modified Kähler classes is an open convex cone called the modified Kähler cone and
denoted by $\MK$. Similarly, we get a closed convex cone $\MN$, the modified nef cone.  If  $X$ is a K\"ahler manifold, $\MN$ is the closure of $\MK$.

\begin{proposition}[Boucksom~\cite{Boucksom}]\label{prop:modified-kahler}
A class $\alpha$ lies in $\MK$ if and only if there exist a modification $\mu:\widetilde X\to X$ and a K\"ahler class $\widetilde\alpha$ on $\widetilde X$ such that $\alpha=\mu_*\widetilde\alpha$.
\end{proposition}

\begin{definition}
    A finite family of prime divisors is called \emph{exceptional} if the convex cone generated by their classes meets $\MN$ only at the origin. 
    \end{definition}
    Boucksom proved that the classes of an exceptional family are linearly independent; see Proposition~3.11(iii) of~\cite{Boucksom}. Consequently, the cardinality of such a family is bounded by the Picard number.

\section{The destabilizing prime divisors for the $J$-equation}
In this section, we first prove Proposition~\ref{keyestimate-intro}, and then derive Theorem~\ref{t1} from it.  To begin with, we prove the following preliminary lemma.  Compared with Proposition~\ref{keyestimate-intro}, it considers the
special case that \(\beta\) is a Kähler class rather than a
modified nef class.

\begin{lemma}\label{estimateforKahlerclass}
  Let $(M^n,\omega)$ be a compact K\"ahler manifold of complex
 dimension $n\geq 2$. Let $\chi$ be another K\"ahler form. Suppose that the $J$-semistable condition ~\eqref{topintegralnonnegative}--~\eqref{$J$-semistablecondition} holds.
Then for any K\"ahler class $\beta$, we have $$\int_M F_{n-1}(\chi)\wedge\beta \geq c\int_M \omega^{n-1}\wedge \beta,$$ where \begin{equation}\label{c=}
  c= \frac{2n\displaystyle\int_M\Big(\chi^n-n\chi^{n-1}\wedge\omega\Big)+\int_M \omega^n}{8n^2\displaystyle\int_M\chi\wedge\omega^{n-1}}.\end{equation}
\end{lemma}
\begin{proof}
We use a mass concentration method, motivated by ~\cite{demailly2004numerical, Chen}. The main idea is to  construct an auxiliary equation based on~\cite{Chen, Song} and estimate the integral by dividing into two subsets.
   
   First note that, by continuity, it suffices to prove that for any K\"ahler class $\beta$ and any $t>1$, we have $$\int_M F_{n-1}(t\chi)\wedge\beta \geq c_t\int_M \omega^{n-1}\wedge \beta,$$ where \begin{equation}\label{c=}
 c_t= \frac{2n\displaystyle\int_M\Big((t\chi)^n-n(t\chi)^{n-1}\wedge\omega\Big)+\int_M \omega^n}{8n^2\displaystyle\int_Mt\chi\wedge\omega^{n-1}}.\end{equation}
 But for any $t>1$, $(t\chi,\omega)$ satisfies the $J$-stable condition in Theorem~\ref{J-stable criterion}  because of the $J$-semistable assumption. So without loss of generality, from now on we may assume that $(\chi,\omega)$ itself satisfies the   $J$-stable condition. 
 
 Then  by Theorem~\ref{J-stable criterion}, we can obtain a K\"ahler form $\chi_\varphi\in[\chi]$, such that 
 \begin{equation}\label{eq:subsolution-inchi}
     \chi_\varphi^{n-1}-(n-1)\chi_\varphi^{n-2}\wedge\omega>0.
 \end{equation}
 Choose a Kähler representative of the class \(\beta\), still
denoted by \(\beta\). Let $C_{\beta}>0$ be a constant determined by \begin{equation}\label{def:Cbeta}
    \int_M\left(\chi^n-n\chi^{n-1}\wedge\omega\right)=\frac{1}{2n}\int_M\left(C_{\beta}\omega^{n-1}\wedge\beta-\omega^n\right).\end{equation}
  Since $$\frac{1}{2n}\frac{(C_{\beta}\omega^{n-1}\wedge\beta-\omega^n)}{\omega^{n}}>-\frac{1}{2n},$$ and the K\"ahler form $\chi_\varphi\in[\chi]$ satisfies~\eqref{eq:subsolution-inchi}, hence by Theorem 1.14 in~\cite{Chen}, the following auxiliary equation  admits a smooth solution $u$:
\begin{equation}\label{auxiliary equation}
\chi_u^n-n\chi_u^{n-1}\wedge\omega=\frac{1}{2n}(C_{\beta}\omega^{n-1}\wedge\beta-\omega^n),
\end{equation} with 
\begin{equation}\label{subsolution of auxiliary eq}
    \chi_u>0,\quad F_{n-1}(\chi_u):=\chi_u^{n-1}-(n-1)\chi_u^{n-2}\wedge\omega>0.
\end{equation}
Let $\lambda=(\lambda_1,,,\lambda_n)$ be the eigenvalue of $\chi_u$ with respect to $\omega$. 
Then at any fixed point, equation~\eqref{auxiliary equation} can be written as 
\begin{equation}\label{pointwiseauxiliaryequation}
    \sigma_n(\lambda)-\sigma_{n-1}(\lambda)=\frac{1}{2n}(C_{\beta}\frac{\omega^{n-1}\wedge\beta}{\omega^n}-1),
\end{equation} 
and the cone condition~\eqref{subsolution of auxiliary eq} can be written as 
\begin{equation}\label{pointwisesubsolution}
    \lambda_i>0,\quad\sigma_{n-1}(\lambda|i)-\sigma_{n-2}(\lambda|i)>0, \quad\text{every} \;1\leq i\leq n.
\end{equation}

Without loss of generality, we may assume that $\lambda_1\geq \lambda_2\geq\cdots\geq\lambda_n$.  We define the set
$$E_{\delta}:=\{x\in M \,|\, \sigma_{n-1}(\lambda|1)-\sigma_{n-2}(\lambda|1)<\delta \},$$
where $\delta>0$ is a constant that will be determined later.
Then on $M\backslash E_{\delta}$, by monotonicity we have $$\sigma_{n-1}(\lambda|i)-\sigma_{n-2}(\lambda|i)\geq \sigma_{n-1}(\lambda|1)-\sigma_{n-2}(\lambda|1)\geq\delta, \quad\text{every}\; 1\leq i\leq n. $$
Hence on $M\backslash E_{\delta}$, $$\Big(\chi_u^{n-1}-(n-1)\chi_u^{n-2}\wedge\omega\Big)\wedge\beta \geq \delta \omega^{n-1}\wedge\beta.$$
Hence \begin{equation}\begin{aligned}\label{mainestimate}
&\int_M F_{n-1}(\chi)\wedge\beta=\int_MF_{n-1}(\chi_u)\wedge\beta\geq\int_{M\backslash E_{\delta}}  F_{n-1}(\chi_u)\wedge\beta\\
    \geq&\int_{M\backslash E_{\delta}}\delta\omega^{n-1}\wedge\beta=\delta\Big(\int_M\omega^{n-1}\wedge\beta- \int_{E_{\delta}}\omega^{n-1}\wedge\beta\Big).
\end{aligned}    
\end{equation}
Next we estimate $\displaystyle\int_{E_{\delta}}\omega^{n-1}\wedge\beta$.
By Equation~\eqref{pointwiseauxiliaryequation}, on $E_{\delta}$ we have
\begin{equation}\begin{aligned}\label{onEdelta}
   \frac{1}{2n}C_{\beta}\frac{\omega^{n-1}\wedge\beta}{\omega^n}&= \lambda_1 \sigma_{n-1}(\lambda|1)-\Big(\lambda_1\sigma_{n-2}(\lambda|1)+\sigma_{n-1}(\lambda|1)\Big)+\frac{1}{2n}\\
   &\leq \delta\lambda_1-\sigma_{n-1}(\lambda|1)+\frac{1}{2n}  \\
  &\leq\delta\lambda_1\leq \delta\sigma_1(\lambda)=n\delta\frac{\chi_u\wedge\omega^{n-1}}{\omega^n}.
\end{aligned}
\end{equation}
The second inequality of~\eqref{onEdelta} uses the fact that $\lambda_i>1$ implied by ~\eqref{pointwisesubsolution}. Integrating ~\eqref{onEdelta} over $E_{\delta}$, we obtain 
\begin{equation}\label{TakeintegralonE}
\begin{aligned}
&\int_{E_{\delta}}\omega^{n-1}\wedge\beta\leq \delta\frac{2n^2}{C_{\beta}}\int_{E_{\delta}}\chi_u\wedge\omega^{n-1}\\
&\leq \delta\frac{2n^2}{C_{\beta}}\int_M \chi_u\wedge\omega^{n-1}=\delta\frac{2n^2}{C_{\beta}}\int_M \chi\wedge\omega^{n-1}.
\end{aligned}
\end{equation}
To estimate ~\eqref{mainestimate}, we choose $\delta>0$ be the constant such that
\begin{equation}
    \delta\frac{2n^2}{C_{\beta}}\int_M \chi\wedge\omega^{n-1}=\frac{1}{2}\int_M\omega^{n-1}\wedge\beta,
    \end{equation}
    that is, \begin{equation}\label{def:delta}    
    \delta=\frac{C_{\beta}}{4n^2}\frac{\displaystyle\int_M\omega^{n-1}\wedge\beta}{\displaystyle\int_M \chi\wedge\omega^{n-1}}.
    \end{equation} 
    Then  inequality~\eqref{TakeintegralonE} implies
    \begin{equation}\label{estimate1/2}
\int_{E_{\delta}}\omega^{n-1}\wedge\beta\leq \frac{1}{2}\int_M \omega^{n-1}\wedge\beta.
\end{equation}
    Combining~\eqref{mainestimate},~\eqref{def:delta} and ~\eqref{estimate1/2}, we obtain  
\begin{equation}\label{finalestimate1}
    \int_M F_{n-1}(\chi)\wedge\beta \geq  \frac{\delta}{2}\int_M\omega^{n-1}\wedge\beta=  \frac{C_{\beta}}{8n^2}\frac{\Big(\displaystyle\int_M\omega^{n-1}\wedge\beta\Big)^2}{\displaystyle\int_M \chi\wedge\omega^{n-1}}.
\end{equation}
Finally, by the definition of $C_{\beta}$, see ~\eqref{def:Cbeta}, we get
\begin{equation}\label{finalestimate2}
    \int_M F_{n-1}(\chi)\wedge\beta \geq    \frac{2n\displaystyle\int_M\Big(\chi^n-n\chi^{n-1}\wedge\omega\Big)+\int_M \omega^n}{8n^2\displaystyle\int_M\chi\wedge\omega^{n-1}} \cdot\displaystyle\Big(\int_M\omega^{n-1}\wedge\beta\Big).
\end{equation}
This proves the Lemma.
\end{proof}

Now we prove Proposition~\ref{keyestimate-intro}.
\begin{proof}[Proof of Proposition~\ref{keyestimate-intro}]
    Since the modified nef cone is the closure of the modified Kähler
cone, and both sides of the desired inequality depend continuously
on \(\beta\), it suffices to prove the estimate when \(\beta\) is a
modified Kähler class.
By Proposition~\ref{prop:modified-kahler}, there exists a modification  $\mu:\widetilde M\to M$ and a K\"ahler class $\widetilde\beta$ on $\widetilde M$ such that $\beta=\mu_*\widetilde\beta$. Let $\alpha$ be a K\"ahler form in the class $\widetilde\beta$.

First we  claim that the K\"ahler pair $(\mu^*(t\chi)+\varepsilon_1\alpha,\mu^*\omega+\varepsilon_2\alpha)$  on $\widetilde M$ is $J$-semistable, for any $t>1$, $\varepsilon_1>0$ and sufficiently small $0<\varepsilon_2<<\text{min}\{t-1,\varepsilon_1\}$ depending on $t,\varepsilon_1$.  By the assumption that $(\chi,\omega)$ is $J$-semistable, for any $t>1$, the pair $(t\chi,\omega)$ on $M$ is $J$-stable, so by Theorem~\ref{J-stable criterion}, there exists a smooth function $u$ such that  $t\chi_u$ is a K\"ahler form and $$(t\chi_u)^{n-1}-(n-1)(t\chi_u)^{n-2}\wedge\omega>0,$$
which implies that  \begin{equation}\label{(p,p)>delta_t}
    (t\chi_u)^p-p(t\chi_u)^{p-1}\wedge\omega>0,\quad \text{for any}\;1\leq p\leq n-1,
    \end{equation} as a $(p,p)$-form in the strong sense. Since $\widetilde M$ is compact and $\alpha$ is a K\"ahler form, we can find a small constant $\delta_t>0$ depending  on $t$, such that 
    \begin{equation}\label{alpha>delta_tchi}
\alpha>\delta_t\mu^*(t\chi_u)\end{equation}
as a $(1,1)$-form.
Denote the $(p,p)$-form $$M_p:=p(\mu^*(t\chi_u)+\varepsilon_1\alpha)^{p-1}\wedge(\varepsilon_2\alpha).$$
Write the $(p,p)$-form
\begin{equation}\label{claimfortchi}
\begin{aligned}
&(\mu^*(t\chi_u)+\varepsilon_1\alpha)^p-p(\mu^*(t\chi_u)+\varepsilon_1\alpha)^{p-1}\wedge(\mu^*\omega+\varepsilon_2\alpha)\\
=\sum&_{j=0}^{p-1}(\varepsilon_1\alpha)^j \wedge\Big(\binom{p}{j}(\mu^*(t\chi_u))^{p-j}-p\binom{p-1}{j}(\mu^*(t\chi_u))^{p-1-j}\wedge(\mu^*\omega)\Big)+(\varepsilon_1\alpha)^p-M_p\\
=\sum&_{j=0}^{p-1}(\varepsilon_1\alpha)^j \wedge\binom{p}{j}\Big((\mu^*(t\chi_u))^{p-j}-(p-j)(\mu^*(t\chi_u))^{p-1-j}\wedge(\mu^*\omega)\Big)+(\varepsilon_1\alpha)^p-M_p
\end{aligned}
\end{equation}
By inequality~\eqref{(p,p)>delta_t}, for every $1\leq p\leq n-1 $, we have
\begin{equation}\label{p,pformpostive}
    \sum_{j=0}^{p-1}(\varepsilon_1\alpha)^j \wedge\binom{p}{j}\Big((\mu^*(t\chi_u))^{p-j}-(p-j)(\mu^*(t\chi_u))^{p-1-j}\wedge(\mu^*\omega)\Big)\geq 0,
    \end{equation}
    Let $0<\varepsilon_2<\mathop{\text{min}}\limits_{1\leq p\leq n}\{\frac{\varepsilon_1^p}{p(\delta_t^{-1}+\varepsilon_1)^{p-1}}\}$, then  by the construction~\eqref{alpha>delta_tchi}, for every $1\leq p\leq n$, we have \begin{equation}\label{epsilon_1alpha)^p-M_p}
    (\varepsilon_1\alpha)^p-M_p=(\varepsilon_1\alpha)^p-p(\mu^*(t\chi_u)+\varepsilon_1\alpha)^{p-1}\wedge(\varepsilon_2\alpha)\geq 0.
    \end{equation}
    Combining ~\eqref{claimfortchi}-\eqref{epsilon_1alpha)^p-M_p},  on any $p$-dimensional subvariety $Z\subset \widetilde M$ with $1\leq p\leq n-1$,  we have
\begin{equation}\label{verifyJnef}
\begin{aligned}
    &\int_Z(\mu^*(t\chi)+\varepsilon_1\alpha)^p-p(\mu^*(t\chi)+\varepsilon_1\alpha)^{p-1}\wedge(\mu^*\omega+\varepsilon_2\alpha)\\
    =&\int_Z(\mu^*(t\chi_u)+\varepsilon_1\alpha)^p-p(\mu^*(t\chi_u)+\varepsilon_1\alpha)^{p-1}\wedge(\mu^*\omega+\varepsilon_2\alpha)\geq 0.
    \end{aligned}
\end{equation}
Combining ~\eqref{claimfortchi}-\eqref{epsilon_1alpha)^p-M_p}, we also obtain
\begin{equation}\label{topintegralonwidetildeM}
\begin{aligned}
&\int_{\widetilde M}(\mu^*(t\chi_u)+\varepsilon_1\alpha)^n-n(\mu^*(t\chi_u)+\varepsilon_1\alpha)^{n-1}\wedge(\mu^*\omega+\varepsilon_2\alpha)\\
   \geq&\int_{\widetilde M}\sum_{j=0}^{n-1}(\varepsilon_1\alpha)^j \wedge\binom{n}{j}\Big((\mu^*(t\chi_u))^{n-j}-(n-j)(\mu^*(t\chi_u))^{n-1-j}\wedge(\mu^*\omega)\Big)\\
    \geq&\int_{\widetilde M} (\mu^*(t\chi))^n-n(\mu^*(t\chi))^{n-1}\wedge\mu^*\omega\\
    =&\int_M (t\chi)^n-n(t\chi)^{n-1}\wedge\omega\geq 0.
    \end{aligned}
    \end{equation}
This proves the claim.

For $t> 1,\varepsilon_1> 0$ and $\varepsilon_2> 0$, we denote $\widetilde\chi_{t,\varepsilon_1}=\mu^*(t\chi)+\varepsilon_1\alpha$, $\widetilde\omega_{\varepsilon_2}=\mu^*\omega+\varepsilon_2\alpha.$ Then by the above claim and by Lemma~\ref{estimateforKahlerclass}, for any $t>1$, $\varepsilon_1>0$ and  $\varepsilon_2=\frac{1}{2}\mathop{\text{min}}\limits_{1\leq p\leq n}\{\frac{\varepsilon_1^p}{p(\delta_t^{-1}+\varepsilon_1)^{p-1}}\}$, we have
\begin{equation}\label{estiamte:t,epsion1,2}
    \int_{\widetilde M}\Big(\widetilde\chi_{t,\varepsilon_1}^{n-1}-(n-1)\widetilde\chi_{t,\varepsilon_1}^{n-2}\wedge\widetilde\omega_{\varepsilon_2} \Big)\wedge\widetilde\beta \geq c_{t,\varepsilon_1,\varepsilon_2}\int_{\widetilde M} \widetilde\omega_{\varepsilon_2}^{n-1}\wedge \widetilde\beta,
    \end{equation}where \begin{equation}\label{c_t=}
 c_{t,\varepsilon_1,\varepsilon_2}= \frac{2n\displaystyle\int_{\widetilde M}\Big(\widetilde\chi_{t,\varepsilon_1}^n-n\widetilde\chi_{t,\varepsilon_1}^{n-1}\wedge\widetilde\omega_{\varepsilon_2}\Big)+\int_{\widetilde M}\widetilde\omega_{\varepsilon_2}^n}{8n^2\displaystyle\int_{\widetilde M} \widetilde\chi_{t,\varepsilon_1}\wedge\widetilde\omega_{\varepsilon_2}^{n-1}}.\end{equation}

    First choose \(t>1\), let \(\varepsilon_1\to0^+\), and then $t\to 1^+$, while $\varepsilon_2=\frac{1}{2}\mathop{\text{min}}\limits_{1\leq p\leq n}\{\frac{\varepsilon_1^p}{p(\delta_t^{-1}+\varepsilon_1)^{p-1}}\}\to 0^+$, then by the continuity of integral, ~\eqref{estiamte:t,epsion1,2} implies
\begin{equation}
\begin{aligned}
   \int_{M}\Big(\chi^{n-1}-(n-1)\chi^{n-2}\wedge\omega \Big)\wedge\beta&=\int_{\widetilde M}\Big((\mu^*\chi)^{n-1}-(n-1)(\mu^*\chi)^{n-2}\wedge(\mu^*\omega) \Big)\wedge\widetilde\beta \\ &\geq c\int_{\widetilde M} (\mu^*\omega)^{n-1}\wedge \widetilde\beta=c\int_M\omega^{n-1}\wedge\beta.
\end{aligned}
\end{equation}

where $$c= \frac{2n\displaystyle\int_{\widetilde M}\Big((\mu^*\chi)^n-n(\mu^*\chi)^{n-1}\wedge\mu^*\omega\Big)+\int_{\widetilde M}(\mu^*\omega)^n}{8n^2\displaystyle\int_{\widetilde M}\mu^*\chi\wedge(\mu^*\omega)^{n-1}} =\frac{2n\displaystyle\int_M\Big(\chi^n-n\chi^{n-1}\wedge\omega\Big)+\int_M \omega^n}{8n^2\displaystyle\int_M\chi\wedge\omega^{n-1}}.$$
This completes the proof.
    \end{proof}

Using the estimate in Proposition~\ref{keyestimate-intro}, now we prove Theorem~\ref{t1}. We restate it here for convenience.
   \begin{theorem}
    \label{mainthm:exceptionalfamily}
 Let $(M^n,\omega)$ be a compact K\"ahler manifold of complex
 dimension $n\geq 2$. Let $\chi$ be another K\"ahler form.  Suppose that the $J$-semistable condition ~\eqref{topintegralnonnegative}--~\eqref{$J$-semistablecondition} holds. Set the positive constant $$c= \frac{2n\displaystyle\int_M\Big(\chi^n-n\chi^{n-1}\wedge\omega\Big)+\int_M \omega^n}{8n^2\displaystyle\int_M\chi\wedge\omega^{n-1}},$$ and let \begin{equation}
     S:=\{ D \,|\,D\; \text{is a prime divisor and } \int_DF_{n-1}(\chi)< c\int_D\omega^{n-1} \}.
 \end{equation}
Then any finite subset of  $S$ is an exceptional family  in the sense of Boucksom~\cite{Boucksom}. In particular, the cardinality of  $S$ is bounded by  the Picard number of $M$.
\end{theorem}
\begin{proof}
   Suppose that  $T:=\{D_1,\ldots,D_j\}$ is a finite subset in $S$, and suppose that $$\alpha=\sum_{i=1}^j x_i\{D_i\}\neq0,\quad x_i\geq 0$$ be a nonzero $(1,1)$-class which is a nonnegative combination of the divisorial classes.  For every \(i\), by the definition of $S$, 
\[
x_i\int_MF_{n-1}(\chi)\wedge\{D_i\}
\leq
cx_i\int_M\omega^{n-1}\wedge\{D_i\},
\]
and the inequality is strict whenever \(x_i>0\).
Since at least one \(x_i\) is positive, summing over \(i\) gives
the strict inequality 
\[
\int_MF_{n-1}(\chi)\wedge\alpha
<
c\int_M\omega^{n-1}\wedge\alpha.
\]
By Proposition~\ref{keyestimate-intro}, $\alpha$ cannot be a modified nef class.
This implies that the convex cone generated by the divisorial classes of the prime divisors in $T$ meets the modified nef cone only at the origin. This  exactly means that $T$ is an exceptional family in the sense of Boucksom~\cite{Boucksom}. 

As a consequence, by Proposition~3.11(iii) of~\cite{Boucksom},  the divisorial classes \(\{D_1\},,,\{D_j\}\) are linearly independent. In particular, the cardinality of  $T$ is bounded by  the Picard number of $M$. By the arbitrary choice of the  finite subset  $T$,  the cardinality of  $S$ is bounded by  the Picard number of $M$. This completes the proof.
\end{proof}

The following corollary of Theorem~\ref{t1} may be useful. 
\begin{corollary}\label{coro:gap-away-nulldivisor}
  Let $(M^n,\omega)$ be a compact K\"ahler manifold of complex
 dimension $n\geq 2$. Let $\chi$ be another K\"ahler form.  Suppose that the $J$-semistable condition ~\eqref{topintegralnonnegative}--~\eqref{$J$-semistablecondition} holds. Then the set 
 \[
\operatorname{Dest}_{n-1}(\chi,\omega)
:=
\bigcup_{\substack{
Z\subsetneq M\ \mathrm{irreducible}\\
\dim Z= n-1,\; \int_ZF_{n-1}(\chi)=0
}}Z
\]
is an analytic subvariety of $M$. And there exists a uniform constant $\delta>0$ independent of the prime divisors, such that for any prime divisor $Y$  satisfying $\int_{Y} F_{n-1}(\chi)>0$, we have 
$$\int_{Y} F_{n-1}(\chi)\geq \delta\int_{Y}\omega^{n-1}.$$
\end{corollary}
\begin{proof}
 Let \begin{equation}
     K:=\{ D \,|\,D\; \text{is a prime divisor and } \int_DF_{n-1}(\chi)=0 \}.
 \end{equation}
 Since $K\subset S$, where $S$ is defined as~\eqref{S}, Theorem~\ref{t1} implies that $K$ is a finite set. Write 
 \(
     K:=\{ D_1,\ldots, D_m\}.
 \)
Then $$\operatorname{Dest}_{n-1}(\chi,\omega)= \bigcup_{i=1}^m D_i,$$  hence it is an analytic subvariety. 

Next we prove the existence of the gap $\delta>0$.  Let \begin{equation}
     S_{+}:=\{ D \in S\,|\,D\; \text{is a prime divisor and } \int_DF_{n-1}(\chi)>0 \}.
 \end{equation}
  First suppose that $S_{+}=\varnothing$,  it means that if a prime divisor $D$ satisfies $\int_DF_{n-1}(\chi)>0$, then $D \notin S$. Hence we have 
$$\int_{D} F_{n-1}(\chi)\geq c\int_{D}\omega^{n-1}$$ by the definition of $S$. Let $\delta=c$ then we are done. Now suppose that $S_{+}\neq\varnothing$, since $S_{+}\subset S$ is a finite set, there is a constant $b>0$, such that $$b=\mathop{\text{inf}}\limits_{D\in S_{+}}\frac{\int_DF_{n-1}(\chi)}{\int_{D}\omega^{n-1}}.$$ 
Clearly $b\leq c$. Let $\delta=b$,  then it clearly  implies the conclusion. 
\end{proof}

\section{Bounded potential solution of the $J$-equation under the smooth boundary cone condition}
Throughout this section, we work under the assumptions of
Theorem~\ref{thm:regularity of the weak solution}. Let
\((M^n,\omega)\) be a compact K\"ahler manifold of complex dimension
\(n\geq3\), and let \(\chi\) be another K\"ahler form. Suppose that
the normalization~\eqref{normalizationcondition} and the
smooth boundary cone condition~\eqref{semisubsolutioncondition} hold.
Let \(\{D_1,\ldots,D_m\}\) be the set of all prime divisors satisfying
\begin{equation}
\int_{D_i}F_{n-1}(\chi)=0.
\end{equation}
If \(m=0\), then the numerical conditions are strict, and the
existence of a smooth solution follows directly from the
\(J\)-stable criterion. Hence, we assume that \(m\geq1\).

The normalization~\eqref{normalizationcondition} and the
smooth boundary cone condition~\eqref{semisubsolutioncondition} exactly coincide with  the smooth boundary cone setting of~\cite{Liu2026Boundary}. Moreover, Proposition~\ref{keyestimate-intro} implies that, for every nonzero
modified nef class \(\xi\),
\begin{equation}\label{eq:numerical-J-bigness-criterion}
\int_M F_{n-1}(\chi)\wedge\xi
\geq
c\int_M\omega^{n-1}\wedge\xi
>0.
\end{equation}
    Thus the numerical \(J\)-bigness criterion of
\cite[Proposition~29]{Liu2026Boundary} is satisfied. Hence the \(J\)-big assumption 
in ~\cite{Liu2026Boundary} can be removed.

Consequently, for \(n\geq4\), we can apply
\cite[Proposition~25, Corollary~26, Lemmas~27 and~28]{Liu2026Boundary}. It follows that, for $n\geq 4$, there exist constants \(a_i>0\),
and   a K\"ahler form \(A\), $A\in[\chi-\sum_{i=1}^m a_i\{D_i\}-\varepsilon\omega]$, such that
\(
F_{n-1}(A)
=A^{\,n-1}-(n-1)A^{\,n-2}\wedge\omega>0.\) For \(n=3\), the smooth boundary cone condition implies
\(F_1(\chi)>0\), so there are no \(J\)-null curves, and the same
conclusion  follows from~\cite{Liu2026Semistable}.

Let \(\widehat\chi:=A+\varepsilon\omega\). Then \(\widehat\chi\) is a K\"ahler form, $\widehat\chi\in[\chi-\sum_{i=1}^m a_i\{D_i\}]$, and satisfying\begin{equation}\label{eq:Section-strict-barrier}
F_{n-1}(\widehat\chi)
=\widehat\chi^{\,n-1}-(n-1)\widehat\chi^{\,n-2}\wedge\omega>0,
\end{equation}
 by the monotonicity of the cone condition.

Let  $s_i$ be the canonical sections of $\mathcal O(D_i)$ vanishing on $D_i$, $h_i$ be the Hermitian metrics on $\mathcal O(D_i)$, and $\theta_{D_i}$ be the normalized curvature forms with respect to $h_i$.
Then we can write
$$\widehat{\chi}=\chi-\sum_{i=1}^m a_i \theta_{D_i}+\sqrt{-1}\partial\bar{\partial}\varphi$$ for
a smooth function \(\varphi\). 
Define the quasi-plurisubharmonic function
\begin{equation}\label{eq:Section-rho}
\rho=\varphi+\sum_{i=1}^m a_i\log\|s_i\|_{h_i}^2.
\end{equation}
By the Lelong-Poincar\'e formula~\cite{Demailly},
\[\chi+\sqrt{-1}\partial\bar\partial\rho
=\widehat\chi+\sum_{i=1}^m a_i[D_i]\geq\widehat\chi,\] 
in the sense of currents. 
In particular,
\begin{equation}\label{eq:J-equation-barrier-off-Z}
\chi+\sqrt{-1}\partial\bar\partial\rho=\widehat\chi
\qquad\text{on }M\setminus Z,\qquad Z=\bigcup_{i=1}^m D_i
\end{equation}
We normalize \(\sup_M\rho=0\). Since every \(a_i>0\), we have
\begin{equation}\label{eq:Section-rho-minus-infty}
\rho(x)\longrightarrow-\infty\qquad\text{as }x\longrightarrow Z.
\end{equation}

Recall that the usual pointwise property of the $J$-cone
implies
\begin{equation}\label{eq:Section-descending-cone}
F_{n-1}(\gamma)\geq0,\ \gamma>0
\quad\Longrightarrow\quad
F_p(\gamma)>0,\qquad 1\leq p\leq n-2.
\end{equation}
This is the standard  property of the $J$-cone; see, for
example, \cite{SongWeinkove,Chen}.

\subsection{The approximating twisted $J$-equations}

For \(t>0\), let \(\chi_t:=\chi+t\omega\) and define
\begin{equation}\label{eq:Section-ct}
c_t:=\frac{\int_MF_n(\chi_t)}
{\int_M\omega^n}.
\end{equation}
The normalization~\eqref{normalizationcondition} and the smooth boundary cone condition~\eqref{semisubsolutioncondition} implies
\begin{equation}\label{eq:Section-ct-estimate}
c_t>0,\qquad c_t\longrightarrow0,
\qquad 0<c_t\leq Ct
\end{equation}
for \(0<t\leq t_0\), after fixing \(t_0>0\) sufficiently small.

We now consider the approximate twisted $J$-equation. 
By the Hessian quotient theory of
Sz\'ekelyhidi~\cite{Szekelyhidi2018}, or equivalently by
Gao Chen's existence theorem for the twisted $J$-equation
\cite[Theorem~1.14]{Chen},  for every $t>0$, there exists a smooth function \(u_t\) satisfying
\begin{equation}\label{eq:Section-approx-equation}
\eta_t:=\chi+t\omega+\sqrt{-1}\partial\bar\partial u_t>0,
\qquad
\eta_t^n=n\eta_t^{n-1}\wedge\omega+c_t\omega^n,
\qquad
\sup_Mu_t=0.
\end{equation}

If \(\lambda_1,\ldots,\lambda_n\) are the eigenvalues of \(\eta_t\)
with respect to \(\omega\), then 
\begin{equation}\label{eq:Section-eigenvalue-equation}
1=\sum_{i=1}^n\frac1{\lambda_i}
+\frac{c_t}{\lambda_1\cdots\lambda_n}.
\end{equation}
Since \(c_t>0\), we have \(\lambda_i>1\) for every \(i\), and hence
\begin{equation}\label{eq:Section-eta-lower}
\eta_t>\omega.
\end{equation}

Our aim is to obtain the estimates of $u_t$ that are uniform in \(t\), and then pass to
the limit as \(t\to0\).  To the best of our knowledge, the global $C^0$ estimate is new in this setting.

The proof of higher regularity on $M\setminus Z$ is an elliptic analogue to the proof of the local regularity of $J$-flow in boundary case; see~\cite{SongWeinkove, Liu2026Semistable} for instance.  For completeness, we also include the details here.

\subsection{Local $C^0$ estimate}

The local \(C^0\) estimate on \(M\setminus Z\) is based on the ABP argument used by
Sz\'ekelyhidi in~\cite[Propositions~10 and~11]{Szekelyhidi2018}. In the present
setting, the relevant domain is a connected component of a sublevel
set of \(u_t-\rho\), which need not be convex and may have irregular
boundary. We shall use the standard lower-contact-set form of the ABP
estimate on bounded domains. This version requires neither convexity
nor boundary regularity and follows directly from the normal-map
argument in~\cite[Lemma~9.2]{GT01}; see also
\cite[Chapter~3]{CaffarelliCabre1995}.

We also need the following elementary pointwise lemma.

\begin{lemma}\label{lem:Section-bounded-set}
There exist constants \(\varepsilon_0,R>0\), independent of
sufficiently small \(t>0\), such that if \(Y>0\) satisfies
\[
Y^n=nY^{n-1}\wedge\omega+c_t\omega^n,
\qquad
Y\geq\widehat\chi+t\omega-\varepsilon_0\omega,
\]
then \(Y\leq R\omega\).
\end{lemma}

\begin{proof}
Since $\widehat{\chi}$ is a K\"ahler form, and \(F_{n-1}(\widehat\chi)>0\), after decreasing
\(\varepsilon_0\) and \(t_0\) if necessary there are fixed
\(\kappa,\delta>0\) such that, for \(0<t\leq t_0\),
\[
\widehat\chi+(t-\varepsilon_0)\omega\geq\kappa\omega,
\qquad
F_{n-1}\bigl(\widehat\chi+(t-\varepsilon_0)\omega\bigr)
\geq\delta\omega^{n-1}.
\]
Thus, for every complex hyperplane
\(H\subset T_x^{1,0}M\), there exists a uniform \(\tau>0\) such that
\begin{equation}\label{eq:Section-hyperplane-gap}
\operatorname{tr}_{(\widehat\chi+(t-\varepsilon_0)\omega)|_H}
\omega|_H\leq1-\tau.
\end{equation}

Choose \(\omega\)-unitary coordinates in which
\(Y=\sum_i\lambda_i e_i\). Since
\(Y\geq\widehat\chi+(t-\varepsilon_0)\omega\), restriction to
hyperplanes give
\[
\sum_{j\neq i}\frac1{\lambda_j}\leq1-\tau,
\qquad
\lambda_j\geq\kappa.
\]
If \(\lambda_i=\max_j\lambda_j\), the equation for \(Y\) gives
\[
1\leq1-\tau+\frac1{\lambda_i}
+\frac{c_t}{\kappa^{n-1}\lambda_i}.
\]
Hence
\[
\lambda_i\leq
\frac{1+c_t\kappa^{-(n-1)}}{\tau}\leq R,
\]
which proves the lemma.
\end{proof}

\begin{proposition}[Local $C^0$ estimate]
\label{prop:Section-local-C0}
There exists a constant \(C>0\), independent of \(t\), such that
\begin{equation}\label{eq:Section-relative-C0}
u_t\geq\rho-C\qquad\text{on }M\setminus Z.
\end{equation}
Consequently, for every \(K\Subset M\setminus Z\), there exists
\(C_K>0\), independent of \(t\), such that
\begin{equation}\label{eq:Section-local-C0}
-C_K\leq u_t\leq0\qquad\text{on }K.
\end{equation}
\end{proposition}

\begin{proof}
Set \(v_t:=u_t-\rho\) on \(M\setminus Z\). By
\eqref{eq:Section-rho-minus-infty},
\(v_t\to+\infty\) as \(x\to Z\), while by~\eqref{eq:J-equation-barrier-off-Z},
\begin{equation}\label{eq:Section-eta-v}
\eta_t=\widehat\chi+t\omega
+\sqrt{-1}\partial\bar\partial v_t
\qquad\text{on }M\setminus Z.
\end{equation}

We first recall a standard uniform \(L^1\)-estimate. By
\eqref{eq:Section-eta-lower},
\[
\sqrt{-1}\partial\bar\partial u_t
=\eta_t-\chi-t\omega\geq(1-t)\omega-\chi\geq-C\omega.
\]
Together with \(\sup_Mu_t=0\), the standard compactness estimate for
normalized quasi-plurisubharmonic functions implies
\begin{equation}\label{eq:Section-L1}
\int_M|u_t|\,\omega^n\leq C,\qquad
\int_M|v_t|\,\omega^n\leq C;
\end{equation}
see, for example, \cite{Demailly}.

Suppose that the desired lower bound fails. Then there exists a
sequence \(t_j\in(0,t_0]\) such that
\[
    m_{t_j}:=\inf_{M\setminus Z}v_{t_j}\longrightarrow-\infty.
\]
For notational simplicity, we relabel \(t_j\) by \(t\) and \(m_{t_j}\)
by \(m_t\).
 Since for every $0<t\leq t_0$,
\(v_t\to+\infty\) near \(Z\), the minimum is attained at some
\(p_t\in M\setminus Z\).
By compactness of \(M\), we may fix finitely many holomorphic
coordinate charts
\[
V_\alpha\Subset U_\alpha,\qquad 1\leq\alpha\leq N,
\]
such that the \(V_\alpha\) cover \(M\). After shrinking the
\(V_\alpha\) if necessary, there exists a constant \(r>0\),
independent of \(\alpha\), with the following property: for every
\(p\in V_\alpha\), after translating the holomorphic coordinates so
that \(z(p)=0\), the Euclidean ball \(B_{2r}(0)\) is contained in the
coordinate image of \(U_\alpha\). Moreover, on these finitely many
charts the Euclidean metric and \(\omega\) are uniformly equivalent,
with all constants independent of \(\alpha\).

For each \(t\), choose an index \(\alpha(t)\) such that
\(p_t\in V_{\alpha(t)}\), and use the corresponding coordinates,
centered at \(p_t\). On \(B_r(0)\), set
\[
w_t(z):=v_t(z)+\varepsilon r^{-2}|z|^2,
\]
where \(\varepsilon>0\) is fixed sufficiently small and independent
of \(t\). Since \(z(p_t)=0\) and
\(m_t=v_t(p_t)=\inf_{M\setminus Z}v_t\), we have
\[
w_t(p_t)=m_t,
\qquad
w_t\geq m_t+\varepsilon
\quad\text{on }\partial B_r(0).
\]
Let \(\Omega_t\) be the connected component containing \(p_t\) of
\[
\{w_t<m_t+\varepsilon\}\cap(B_r(0)\setminus Z).
\]
Since \(v_t(x)\to+\infty\) as \(x\to Z\), the sublevel set
\(\{w_t<m_t+\varepsilon\}\) stays a positive distance away from
\(Z\), hence
\(
    \overline{\Omega_t}\cap Z=\varnothing.
\)
Moreover,
\(
\operatorname{diam}_{\mathbb R}(\Omega_t)\leq 2r,
\)
and
\[
w_t\geq m_t+\varepsilon
\qquad\text{on }\partial\Omega_t.
\]
Indeed, on the portion of \(\partial\Omega_t\) contained in the
interior of \(B_r(0)\), equality
\(w_t=m_t+\varepsilon\) holds by the definition of the connected
component, while on
\(\partial\Omega_t\cap\partial B_r(0)\) the same lower bound follows
from \(v_t\geq m_t\).

The bounded-domain ABP contact-set estimate
\cite[Lemma 9.2]{GT01}
(cf. \cite[Proposition 11]{Szekelyhidi2018})
applied in real dimension \(2n\),
then gives a contact set \(\Gamma_t\subset\Omega_t\) such that
\begin{equation}\label{eq:Section-ABP-contact}
D_{\mathbb R}^2w_t\geq0,\qquad
w_t<m_t+\frac{\varepsilon}{2}
\quad\text{on }\Gamma_t,
\end{equation}
and
\begin{equation}\label{eq:Section-ABP-integral}
|B_1(0)|
\left(\frac{\varepsilon}{4r}\right)^{2n}
\leq
\int_{\Gamma_t}
\det_{\mathbb R}\bigl(D_{\mathbb R}^2w_t\bigr)\,dx.
\end{equation}

The first inequality in \eqref{eq:Section-ABP-contact} implies
\(\sqrt{-1}\partial\bar\partial v_t\geq-C\varepsilon\omega\)
on \(\Gamma_t\). Choosing \(\varepsilon\) so that
\(C\varepsilon\leq\varepsilon_0\), Lemma~\ref{lem:Section-bounded-set}
and \eqref{eq:Section-eta-v} imply
\begin{equation}\label{eq:Section-eta-contact}
\eta_t\leq R\omega\qquad\text{on }\Gamma_t.
\end{equation}
Thus the complex Hessian of \(w_t\) is uniformly bounded above on
\(\Gamma_t\), independent of $t$. Since its real Hessian is nonnegative there, for every
real vector \(\xi\),
\[
0\leq D_{\mathbb R}^2w_t(\xi,\xi)
\leq D_{\mathbb R}^2w_t(\xi,\xi)
+D_{\mathbb R}^2w_t(J\xi,J\xi)\leq C|\xi|^2.
\]
Hence
\(\det_{\mathbb R}D_{\mathbb R}^2w_t\leq C\) on \(\Gamma_t\).
Combining this with \eqref{eq:Section-ABP-integral} gives
\begin{equation}\label{eq:Section-contact-volume}
\operatorname{Vol}_\omega(\Gamma_t)\geq c_0>0.
\end{equation}

On the other hand, \eqref{eq:Section-ABP-contact} gives
\(v_t<m_t+\varepsilon/2\) on \(\Gamma_t\). Since \(m_t\to-\infty\), for large \(|m_t|\),
\[
\Gamma_t\subset
\left\{|v_t|\geq|m_t|-\frac{\varepsilon}{2}\right\}.
\]
By Chebyshev's inequality and \eqref{eq:Section-L1},
\[
\operatorname{Vol}_\omega(\Gamma_t)
\leq\frac{C}{|m_t|-\varepsilon/2}\longrightarrow0,
\]
contradicting \eqref{eq:Section-contact-volume}. Hence
\(\inf_{M\setminus Z}v_t\geq-C\), which is
\eqref{eq:Section-relative-C0}. The local estimate
\eqref{eq:Section-local-C0} follows immediately because \(\rho\)
is smooth on \(M\setminus Z\) and \(u_t\leq0\).
\end{proof}

\subsection{Uniform global $C^0$ estimate}\label{Section:Uniform-global-$C^0$ estimate}

 We now upgrade the local estimate to a uniform global \(C^0\)
estimate. Inspired by Jeffres' trick~\cite{Jeffres2000}, we
replace the logarithmic barrier \(\rho\) by the bounded barrier
\[
\psi=e^{a\rho},
\]
where \(a>0\) is chosen sufficiently small. To the
best of our knowledge, this is the first application of such a
Jeffres-type barrier to the \(J\)-equation.  It has two essential advantages: its slow vanishing prevents
the minimum of \(u_t-\psi\) from attaining on \(Z\), while the blow-up of
\[
\frac{F_n\bigl(\chi+\sqrt{-1}\partial\bar\partial\psi\bigr)}
{\omega^n}
\]
near \(Z\) forces the minimum point to remain outside a fixed
neighborhood of \(Z\).

We first recall an elementary quantitative property of the
$J$-cone.

\begin{lemma}\label{lem:Section-interpolation}
Let \(\gamma_0,\gamma_1\) be K\"ahler forms satisfying
\(F_{n-1}(\gamma_0)\geq0\) and \(F_{n-1}(\gamma_1)>0\). Then, for
\(B_\lambda=(1-\lambda)\gamma_0+\lambda\gamma_1\),
\(0\leq\lambda\leq1\), there exists \(\delta>0\) such that
\begin{equation}\label{eq:Section-interpolation}
F_{n-1}(B_\lambda)\geq
\delta\lambda\,\omega^{n-1}.
\end{equation}
\end{lemma}

\begin{proof}
We use the convexity of the closed \(J\)-cone; see
\cite[Remark~3.4]{Chen}. In the notation of \cite[Section~3]{Chen},
the cone condition is characterized by
\[
F_{n-1}(\gamma)\geq0
\quad\Longleftrightarrow\quad
P_\omega(\gamma)\leq1,
\]
where \(P_\omega\) is convex.
Since \(F_{n-1}(\gamma_1)>0\) and \(M\) is compact, there exists
\(\delta_0>0\) such that
\(
P_\omega(\gamma_1)\leq1-\delta_0.
\)
By the convexity of \(P_\omega\),
\[
P_\omega(B_\lambda)
\leq
(1-\lambda)P_\omega(\gamma_0)
+\lambda P_\omega(\gamma_1)
\leq
1-\delta_0\lambda.
\]
Using again the characterization of the \(J\)-cone, we obtain
\[
(1-\delta_0\lambda)B_\lambda^{n-1}
-(n-1)B_\lambda^{n-2}\wedge\omega
\geq0,
\]
and hence
\[
F_{n-1}(B_\lambda)
\geq
\delta_0\lambda B_\lambda^{n-1}.
\]
Since \(\gamma_0\) and \(\gamma_1\) are K\"ahler forms on the compact
manifold \(M\), there exists \(c>0\) such that
\(B_\lambda\geq c\omega\) for every \(0\leq\lambda\leq1\).
Therefore
\[
F_{n-1}(B_\lambda)
\geq
\delta_0c^{n-1}\lambda\,\omega^{n-1}.
\]
The conclusion follows.
\end{proof}

\begin{lemma}[Bounded divisorial barrier]
\label{lem:bounded-barrier}
There exists \(a>0\) sufficiently small such that
\[
\psi:=e^{a\rho}\quad\text{on }M\setminus Z,
\qquad
\psi|_Z:=0,
\]
satisfies
\[
\psi\in C^0(M)\cap C^\infty(M\setminus Z),
\quad 0\leq\psi\leq 1,\quad\widetilde\chi
:=
\chi+\sqrt{-1}\partial\bar\partial\psi>0,
\]
\begin{equation}\label{eq:J-flow-barrier-cone}
F_{n-1}(\widetilde\chi)>0
\quad\text{on }M\setminus Z,
\end{equation}
and
\begin{equation}\label{eq:J-equation-barrier-blow-up}
\frac{F_n(\widetilde\chi)}{\omega^n}
\longrightarrow+\infty
\qquad\text{as }x\longrightarrow Z.
\end{equation}
Finally, for every \(u\in C^\infty(M)\), the function \(u-\psi\)
cannot attain a local minimum at a point of \(Z\).
\end{lemma}

\begin{proof}
Since \(\rho\leq0\) and \(\rho\to-\infty\) along \(Z\), the extension
\(\psi|_Z=0\) is continuous and \(0\leq\psi\leq1\).

Set \(\lambda=ae^{a\rho}\). We choose \(a<1\) such that
\(0\leq\lambda<1\). Using
\eqref{eq:J-equation-barrier-off-Z}, on \(M\setminus Z\) we have
\begin{equation}\label{eq:J-flow-barrier-decomposition}
\widetilde\chi
=
(1-\lambda)\chi+\lambda\widehat\chi
+a\lambda\sqrt{-1}\partial\rho\wedge\bar\partial\rho.
\end{equation}
Let
\[
B_\lambda:=(1-\lambda)\chi+\lambda\widehat\chi,
\qquad
R:=a\lambda\sqrt{-1}\partial\rho\wedge\bar\partial\rho.
\]
Then \(R\geq0\), \(R^2=0\), and
\(\widetilde\chi=B_\lambda+R\). In particular,
\(\widetilde\chi>0\).

Since \(F_{n-1}(\chi)\geq0\) by the smooth boundary cone condition  and
\(F_{n-1}(\widehat\chi)>0\) by ~\eqref{eq:Section-strict-barrier}, Lemma~\ref{lem:Section-interpolation}
implies
\begin{equation}\label{eq:F_n-1-B>=}
F_{n-1}(B_\lambda)
\geq\delta\lambda\,\omega^{n-1},\end{equation} and then~\eqref{eq:Section-descending-cone}
implies
\(F_{n-2}(B_\lambda)\geq0\). Since \(R^2=0\), on $M\setminus Z$ we have
\[
F_{n-1}(\widetilde\chi)=F_{n-1}(B_\lambda+R)
=
F_{n-1}(B_\lambda)
+(n-1)R\wedge F_{n-2}(B_\lambda)>0,
\]
which proves \eqref{eq:J-flow-barrier-cone}. A similar computation gives
\begin{equation}\label{eq:F_n-expansion}
F_n(\widetilde\chi)
=
F_n(B_\lambda)+nR\wedge F_{n-1}(B_\lambda).
\end{equation}
The forms \(B_\lambda\), \(0\leq\lambda\leq1\), form a compact family
of K\"ahler forms, so \(F_n(B_\lambda)\geq-C\omega^n\). Combining this with ~\eqref{eq:F_n-1-B>=}-~\eqref{eq:F_n-expansion}, we have
\begin{equation}\label{eq:J-flow-weighted-gradient}
\frac{F_n(\widetilde\chi)}{\omega^n}
\geq
-C+c_a e^{2a\rho}|\partial\rho|_\omega^2
\end{equation}
for a constant \(c_a>0\). We now prove that
\[
e^{2a\rho}|\partial\rho|_\omega^2\longrightarrow+\infty
\qquad\text{as }x\longrightarrow Z.
\]
Let \(\pi:\widetilde M\to M\) be a log resolution of \(Z\), which is
an isomorphism over \(M\setminus Z\), such that
\(\pi^{-1}(Z)\) has simple normal crossing support; see
Hironaka~\cite{Hironaka1977}. Fix a K\"ahler metric
\(\widetilde\omega\) on \(\widetilde M\).

Around every point of \(\pi^{-1}(Z)\), choose holomorphic coordinates
\(w=(w_1,\ldots,w_n)\) such that
\[
\pi^{-1}(Z)=\{w_1\cdots w_k=0\}
\]
as a set. Recall
\[
\rho=\varphi+\sum_{i=1}^m a_i\log\|s_i\|_{h_i}^2,
\qquad a_i>0.
\]
Since the pullback of each defining section can be written locally as
\(
\pi^*s_i=u_i\prod_{j=1}^k w_j^{m_{ij}},
\)
where \(u_i\) is nowhere vanishing and \(m_{ij}\geq0\), therefore
\begin{equation}\label{eq:J-flow-rho-resolution}
\pi^*\rho
=
\sum_{j=1}^k b_j\log|w_j|^2+H,
\qquad
b_j:=\sum_{i=1}^ma_i m_{ij}>0,
\end{equation}
where \(H\) is a  smooth function. The positivity of \(b_j\) follows from the
fact that every component \(\{w_j=0\}\) of \(\pi^{-1}(Z)\) is contained
in \(\pi^{-1}(D_i)\) for some \(i\), so we have \(m_{ij}>0\) for some
\(i\), and hence \(b_j=\sum_i a_i m_{ij}>0\).

Since \(\pi^*\omega\leq C\widetilde\omega\), on
\(\widetilde M\setminus\pi^{-1}(Z)\) we have
\begin{equation}\label{eq:pullback-partialrhp-estimate}
    |\partial\rho|_\omega^2\circ\pi
=
|\partial(\pi^*\rho)|_{\pi^*\omega}^2
\geq C^{-1}
|\partial(\pi^*\rho)|_{\widetilde\omega}^2.
\end{equation}
Shrink the coordinate chart so that \(|w_j|<1\), and set
\(d(w):=\min_{1\leq j\leq k}|w_j|\). Choose \(j_0\) such that
\(|w_{j_0}|=d(w)\). From
\eqref{eq:J-flow-rho-resolution},
\begin{equation*}
\frac{\partial(\pi^*\rho)}{\partial w_{j_0}}
=
\frac{b_{j_0}}{w_{j_0}}
+\frac{\partial H}{\partial w_{j_0}}.
\end{equation*}
Since \(H\) is smooth, after restricting to a sufficiently small
neighborhood of \(\pi^{-1}(Z)\), 
\begin{equation}\label{eq:partialrho>=cd(w)}
|\partial(\pi^*\rho)|_{\widetilde\omega}^2
\geq c\,d(w)^{-2}.
\end{equation}
On the other hand,
\begin{equation}\label{eq:expotential-estimate}
e^{2a\pi^*\rho}
=
e^{2aH}\prod_{j=1}^k|w_j|^{4ab_j}
\geq c\,d(w)^{4aB},
\qquad
B:=\sum_{j=1}^k b_j.
\end{equation}
Combining~\eqref{eq:pullback-partialrhp-estimate}, ~\eqref{eq:partialrho>=cd(w)} and~\eqref{eq:expotential-estimate}, we obtain
\begin{equation}\label{eq:J-flow-resolution-growth}
\left(e^{2a\rho}|\partial\rho|_\omega^2\right)\circ\pi
\geq c\,d(w)^{4aB-2}.
\end{equation}
Since \(\pi^{-1}(Z)\) is compact, it can be covered by finitely many
coordinate charts of the above type. We may therefore choose \(a>0\)
so small that \begin{equation}\label{eq:2aB<1}
    2aB<1
\end{equation} on every such chart. Then
\(4aB-2<0\), and
\eqref{eq:J-flow-resolution-growth} implies
\[
e^{2a\rho}|\partial\rho|_\omega^2
\longrightarrow+\infty
\qquad\text{as }x\longrightarrow Z.
\]
Together with \eqref{eq:J-flow-weighted-gradient}, this proves
\eqref{eq:J-equation-barrier-blow-up}.

We finally prove that the function \(u-\psi\)
cannot attain a local minimum at a point of \(Z\). This is the usual
method in Jeffres' trick~\cite{Jeffres2000}. Suppose, to the
contrary, that for some \(u\in C^\infty(M)\), the function
\(u-\psi\) has a local minimum at a point \(p\in Z\). Choose
\(\widetilde p\in\pi^{-1}(p)\). Then
\(
(u-\psi)\circ\pi
\)
has a local minimum at \(\widetilde p\).

Choose one of the above simple-normal-crossing coordinate charts
centered at \(\widetilde p\), so that
\[
\pi^{-1}(Z)=\{w_1\cdots w_k=0\}.
\]
After shrinking the chart, consider the holomorphic disc
$$\gamma:(\Delta,0)\to(\widetilde M,\widetilde p),\qquad
\gamma(t)=(t,\ldots,t,0,\ldots,0),$$
where the first \(k\) coordinates are equal to \(t\). Then
\(\gamma(0)=\widetilde p\) and
\(\gamma(t)\notin\pi^{-1}(Z)\) for \(t\neq0\). By
\eqref{eq:J-flow-rho-resolution},
\[
\pi^*\rho(\gamma(t))
=
B\log|t|^2+H(\gamma(t)),\qquad B=\sum_{j=1}^k b_j,
\]
and hence 
\begin{equation}\label{eq:J-flow-Jeffres-growth}
e^{a\pi^*\rho(\gamma(t))}
\geq c|t|^{2aB},
\end{equation}
for a constant $c>0$.
Since \(u-\psi\) has a local minimum at \(p\in Z\), for sufficiently small
\(t\neq0\),
\[
u(\pi(\gamma(t)))-\psi(\pi(\gamma(t)))
\geq u(p)-\psi(p)=u(p).
\]
Since \(\psi=e^{a\rho}\) on \(M\setminus Z\), we obtain
\[
u(\pi(\gamma(t)))-u(p)
\geq \psi(\pi(\gamma(t)))=e^{a\pi^*\rho(\gamma(t))}\geq
c|t|^{2aB}.
\]
On the other hand, \(u\circ\pi\circ\gamma\) is smooth, and therefore
\[
|u(\pi(\gamma(t)))-u(p)|\leq C|t|.
\]
We thus obtain \[c|t|^{2aB}\leq C|t|,\] which is impossible as \(t\to0^+\),
because  \(2aB<1\)  by our choice of~\eqref{eq:2aB<1}. This contradiction proves that the function \(u-\psi\)
cannot attain a local minimum at a point of \(Z\). The proof is complete.
\end{proof}

\begin{proposition}[Uniform global $C^0$ estimate]
\label{prop:Section-global-C0}
There exists \(C>0\), independent of \(t\), such that
\begin{equation}\label{eq:Section-global-C0}
-C\leq u_t\leq0\qquad\text{on }M,\qquad 0<t\leq t_0.
\end{equation}
\end{proposition}

\begin{proof}
Let \(\psi\) be given by Lemma~\ref{lem:bounded-barrier}. By \eqref{eq:J-equation-barrier-blow-up}, there is a constant $c_*>\sup_{0<t\leq t_0}c_t$, and a fixed neighborhood
\(U\supset Z\) such that
\begin{equation}\label{eq:Section-fixed-U}
F_n(\widetilde\chi)>c_*\omega^n
\qquad\text{on }U\setminus Z.
\end{equation} Define the compact set
\(K:=M\setminus U\Subset M\setminus Z\). 

Since
\(u_t-\psi\in C^0(M)\), it attains its minimum at some \(x_t\in M\).
By the last assertion of Lemma~\ref{lem:bounded-barrier}, \(x_t\notin Z\);
hence
\[
\sqrt{-1}\partial\bar\partial(u_t-\psi)(x_t)\geq0.
\]At the minimum point \(x_t\), set
\[
    H:=t\omega+\sqrt{-1}\partial\bar\partial(u_t-\psi)\ge0.
\]
A direct expansion at \(x_t\) gives
\[
\begin{aligned}
F_n(\widetilde\chi+H)-F_n(\widetilde\chi)=
\sum_{k=1}^{n-1}\binom nk
H^k\wedge F_{n-k}(\widetilde\chi)+H^n\ge0,
\end{aligned}
\]
where the last inequality follows from 
~\eqref{eq:J-flow-barrier-cone} and~\eqref{eq:Section-descending-cone}. Therefore
\(F_n(\eta_t)(x_t)\geq F_n(\widetilde\chi)(x_t)\).

If \(x_t\in U\setminus Z\), then
\[
c_t\omega^n=F_n(\eta_t)(x_t)
\geq F_n(\widetilde\chi)(x_t)>c_*\omega^n,
\]
a contradiction. Hence \(x_t\in K\), for every $0<t\leq t_0$. By
Proposition~\ref{prop:Section-local-C0},
\(u_t(x_t)\geq-C_K\), where $C_K$ is independent of $t$. Since \(0\leq\psi\leq1\),
\[
u_t(x_t)-\psi(x_t)\geq-C_K-1.
\]
By minimality, \(u_t-\psi\geq-C_K-1\) on \(M\). Together with \(\sup_Mu_t=0\), this proves
\eqref{eq:Section-global-C0}.
\end{proof}

\subsection{Local second-order estimate}

We next establish a uniform second-order estimate on compact subsets
of \(M\setminus Z\). The argument is motivated by the second-order
estimate of Song--Weinkove~\cite[Lemma~3.1]{SongWeinkove} and its
twisted version due to Chen~\cite[Proposition~2.1]{Chen}. The local
trace calculation on \(M\setminus Z\) is almost the same as in
Chen's argument. The point in our situation is that the function
\(v_t:=u_t-\rho\) tends to \(+\infty\) along \(Z\), so that the
maximum-principle argument can be carried out entirely on
\(M\setminus Z\).

\begin{proposition}[Local second-order estimate]
\label{prop:Section-local-C2}
For every compact subset \(K\Subset M\setminus Z\), there exists a
constant \(C_K>0\), independent of \(t\), such that
\begin{equation}\label{eq:Section-local-C2}
\omega<\eta_t\leq C_K\omega
\qquad\text{on }K.
\end{equation}
\end{proposition}

\begin{proof}
Recall that on \(M\setminus Z\),
\(
\eta_t=\widehat\chi+t\omega+\sqrt{-1}\partial\bar\partial v_t.
\)
Since \(F_{n-1}(\widehat\chi)>0\), the forms
\(\widehat\chi+t\omega\) satisfy a uniform strict \(J\)-cone
condition for \(0<t\leq t_0\). Equivalently, there exists
\(\tau>0\), independent of \(t\), such that for every \(x\in M\),
every complex hyperplane \(H\subset T_x^{1,0}M\), and every
\(0<t\leq t_0\),
\begin{equation}\label{eq:Section-uniform-hyperplane-gap}
\operatorname{tr}_{(\widehat\chi+t\omega)|_H}(\omega|_H)
\leq1-\tau.
\end{equation}

Following Chen~\cite[Proposition~2.1]{Chen}, we first prove the
standard trace differential inequality
\begin{equation}\label{eq:Section-log-trace-differential}
\mathcal L_t\log\operatorname{tr}_{\omega}\eta_t\geq-C_0
\qquad\text{on }M\setminus Z,
\end{equation}
where \(C_0\) is independent of \(t\). 
Equation~\eqref{eq:Section-approx-equation} can be written as
\begin{equation}\label{eq:Section-approx-equation-writtenas}
\operatorname{tr}_{\eta_t}\omega
+c_t\frac{\omega^n}{\eta_t^n}=1.
\end{equation}
Fix \(x\in M\setminus Z\), and choose holomorphic coordinates centered
at \(x\) which are normal for \(\omega\) and diagonalize \(\eta_t\)
at \(x\):
\[
\omega_{i\bar j}=\delta_{ij},\qquad
\partial_k\omega_{i\bar j}=0,\qquad
g_{i\bar j}:=(\eta_t)_{i\bar j}=\lambda_i\delta_{ij}.
\]
Set
\[
S:=\operatorname{tr}_{\omega}\eta_t=\sum_i\lambda_i,
\qquad
q:=\frac{c_t}{\lambda_1\cdots\lambda_n}.
\]
Then \(\sum_i\lambda_i^{-1}+q=1\). In particular,
\(0<\lambda_i^{-1}<1\) and \(0\leq q<1\).

Let \(\mathcal L_t\) denote the linearization at \(\eta_t\) of
\[
A\longmapsto-\operatorname{tr}_{A}\omega
-c_t\frac{\omega^n}{A^n}.
\]
At the chosen point,
\[
\mathcal L_t f=\sum_k F^{k\bar k}f_{k\bar k},
\qquad
F^{k\bar k}=\frac1{\lambda_k^2}+\frac{q}{\lambda_k}>0,
\]
so \(\mathcal L_t\) is elliptic.

Since \(S=\omega^{i\bar j}g_{i\bar j}\), normality of the coordinates
gives, at \(x\),
\[
S_k=\sum_i g_{i\bar i,k},
\qquad
S_{k\bar k}
=
\sum_i g_{i\bar i,k\bar k}
+\sum_i\lambda_i(\omega^{i\bar i})_{k\bar k}.
\]
Consequently,
\begin{align*}
\mathcal L_t\log S
={}&
\frac1S\sum_{i,k}F^{k\bar k}g_{i\bar i,k\bar k}
-\frac1{S^2}\sum_kF^{k\bar k}
\left|\sum_i g_{i\bar i,k}\right|^2
+\frac1S\sum_{i,k}
F^{k\bar k}\lambda_i(\omega^{i\bar i})_{k\bar k}.
\end{align*}
The last term is uniformly bounded. Indeed, the second derivatives
of \(\omega^{-1}\) in normal coordinates are controlled by the
curvature of the fixed K\"ahler metric \(\omega\), while
\[
\sum_kF^{k\bar k}
=
\sum_k\lambda_k^{-2}
+q\sum_k\lambda_k^{-1}
\leq
\left(\sum_k\lambda_k^{-1}\right)^2
+q\sum_k\lambda_k^{-1}
=1-q\leq1.
\]
Hence
\[
\left|
\frac1S\sum_{i,k}
F^{k\bar k}\lambda_i(\omega^{i\bar i})_{k\bar k}
\right|\leq C.
\]
It remains to estimate the first two terms above. Differentiating equation~\eqref{eq:Section-approx-equation-writtenas}
twice and using the K\"ahler identities gives exactly the trace
calculation in~\cite[Proposition~2.1]{Chen}. In our situation, the
twisting coefficient \(c_t\) is constant,
so the terms involving first or second derivatives of the twisting
function in Chen's calculation vanish. After absorbing the curvature
terms of the fixed metric \(\omega\) into a uniform constant, one
obtains
\[
\mathcal L_t\log S
\geq
-C
-\frac1{S^2}\sum_kF^{k\bar k}
\left|\sum_i g_{i\bar i,k}\right|^2
+\frac1S\sum_{i,j,k}
\frac{F^{i\bar i}}{\lambda_j}|g_{i\bar j,k}|^2.
\]
 For each \(k\), the Cauchy--Schwarz inequality gives
\[
\left|\sum_i g_{i\bar i,k}\right|^2
\leq
S\sum_i\frac{|g_{i\bar i,k}|^2}{\lambda_i}.
\]
Using the K\"ahler identity \(g_{i\bar i,k}=g_{k\bar i,i}\), we obtain
\[
\frac1{S^2}\sum_kF^{k\bar k}
\left|\sum_i g_{i\bar i,k}\right|^2
\leq
\frac1S\sum_{i,k}
\frac{F^{k\bar k}}{\lambda_i}|g_{i\bar i,k}|^2
=
\frac1S\sum_{i,k}
\frac{F^{k\bar k}}{\lambda_i}|g_{k\bar i,i}|^2
\leq
\frac1S\sum_{i,j,k}
\frac{F^{i\bar i}}{\lambda_j}|g_{i\bar j,k}|^2.
\]
This proves~\eqref{eq:Section-log-trace-differential}.

We now use the singular barrier. Consider the smooth function
\[
G_t:=\log S-A v_t
\qquad\text{on }M\setminus Z,
\]
where \(A>0\) will be chosen below. For each fixed \(t>0\), 
\(v_t=u_t-\rho\to+\infty\) as \(x\to Z\), \(\log S\) is a smooth function on $M$, and hence \(G_t\) attains its maximum
at some point \(x_t\in M\setminus Z\).

We claim that there exist constants \(R,\varepsilon>0\), independent
of \(t\), such that
\begin{equation}\label{eq:Section-Lv-gap}
S(x_t)\geq R
\quad\Longrightarrow\quad
-\mathcal L_t v_t(x_t)\geq\varepsilon.
\end{equation}
To prove the claim, choose at \(x_t\) the coordinates used above and
write
\[
B_t:=\widehat\chi+t\omega,
\qquad
b_i:=(B_t)_{i\bar i}.
\]
For each \(i\), let \(H_i\) be the coordinate hyperplane obtained by
omitting the \(i\)-th direction. By
\eqref{eq:Section-uniform-hyperplane-gap},
\[
\operatorname{tr}_{B_t|_{H_i}}(\omega|_{H_i})\leq1-\tau.
\]
Although \(B_t\) need not be diagonal in these coordinates, for every
positive definite Hermitian matrix \(B\) one has
\((B^{-1})_{j\bar j}\geq \frac{1}{B_{j\bar j}}\). Applying this to the
restriction \(B_t|_{H_i}\) gives
\begin{equation}\label{eq:Section-bi-reciprocal}
\sum_{j\neq i}\frac1{b_j}
\leq
\operatorname{tr}_{B_t|_{H_i}}(\omega|_{H_i})
\leq1-\tau.
\end{equation}
Since \((v_t)_{i\bar i}=\lambda_i-b_i\), we compute
\[
-\mathcal L_t v_t
=
\sum_iF^{i\bar i}(b_i-\lambda_i)
=
\sum_i\frac{b_i}{\lambda_i^2}
+q\sum_i\frac{b_i}{\lambda_i}
-1-(n-1)q.
\]
Choose \(i_0\) so that \(\lambda_{i_0}=\max_i\lambda_i\). 
Dropping the nonnegative term
\(q\sum_i b_i\lambda_i^{-1}\) and applying the weighted
Cauchy--Schwarz inequality together with
\eqref{eq:Section-bi-reciprocal}, we obtain
\[
-\mathcal L_t v_t
\geq
\sum_{j\neq i_0}\frac{b_j}{\lambda_j^2}
-1-(n-1)q
\geq
\frac{\left(\sum_{j\neq i_0}\lambda_j^{-1}\right)^2}
{\sum_{j\neq i_0}b_j^{-1}}
-1-(n-1)q
\geq
\frac{(1-q-\lambda_{i_0}^{-1})^2}{1-\tau}
-1-(n-1)q.
\]
Since
\(
\lambda_{i_0}^{-1}\leq \frac{n}{S},\) and
\(
q\leq C\lambda_{i_0}^{-1}\leq\frac{C}{S},
\) when $S$ is sufficiently large,
we have
\[
\frac{(1-q-\lambda_{i_0}^{-1})^2}{1-\tau}
-1-(n-1)q
\geq
\frac{\tau}{1-\tau}-\frac{C}{S}\geq 
\frac{\tau}{2(1-\tau)},
\]
where \(C\) is independent of \(t\). Hence, choosing \(R\) sufficiently
large,
we obtain~\eqref{eq:Section-Lv-gap}.

We return to the maximum point \(x_t\). Since
\(\mathcal L_tG_t(x_t)\leq0\), the differential inequality
\eqref{eq:Section-log-trace-differential} implies
\[
0\geq\mathcal L_tG_t(x_t)
\geq-C_0-A\mathcal L_t v_t(x_t).
\]
Choose \(A>0\) such that \(A\varepsilon>C_0+1\). If
\(S(x_t)\geq R\), then~\eqref{eq:Section-Lv-gap} gives
\(0\geq-C_0+A\varepsilon>1\), a contradiction. Therefore
\(S(x_t)\leq R\).

By the maximality of \(G_t\), for every \(x\in M\setminus Z\),
\[
\log S(x)-Av_t(x)
\leq
\log R-Av_t(x_t)
\leq
\log R-A\inf_{M\setminus Z}v_t.
\]
Consequently,
\begin{equation}\label{eq:Section-weighted-trace-estimate}
\operatorname{tr}_{\omega}\eta_t
\leq
R\exp\!\left(
A\bigl(v_t-\inf_{M\setminus Z}v_t\bigr)
\right)
\qquad\text{on }M\setminus Z.
\end{equation}

By Proposition~\ref{prop:Section-local-C0},
\(v_t=u_t-\rho\geq-C\) on \(M\setminus Z\), with \(C\) independent
of \(t\). If \(K\Subset M\setminus Z\), then \(\rho\) is bounded on
\(K\) and \(u_t\leq0\), hence \(v_t\leq C_K\) on \(K\). It follows
from~\eqref{eq:Section-weighted-trace-estimate} that
\(\operatorname{tr}_{\omega}\eta_t\leq C_K\) on \(K\), and therefore
\(\eta_t\leq C_K\omega\). Combining this with the  lower bound
\(\eta_t>\omega\) from~\eqref{eq:Section-eta-lower} proves
\eqref{eq:Section-local-C2}.
\end{proof}

\subsection{Local higher-order estimates}

The local second-order estimate shows that, on every compact subset
of \(M\setminus Z\), the eigenvalues of \(\eta_t\) with respect to
\(\omega\) remain in a fixed compact subset of the positive cone,
independently of \(t\). We can therefore regard the twisted
\(J\)-equation as a uniformly elliptic concave equation there.
The complex Evans--Krylov theory then gives a genuine
\(C^{2,\alpha}\)-estimate, and the higher-order estimates follow by
the usual Schauder bootstrapping. We refer to
\cite{Szekelyhidi2018} for this general regularity work for
fully nonlinear equations on Hermitian manifolds.

\begin{proposition}[Local higher-order estimates]
\label{prop:Section-higher-order}
For every compact subset \(K\Subset M\setminus Z\) and every integer
\(k\geq0\), there exists a constant \(C_{K,k}>0\), independent of
\(t\), such that
\begin{equation}\label{eq:Section-higher-order}
\|u_t\|_{C^k(K)}\leq C_{K,k},
\qquad 0<t\leq t_0.
\end{equation}
\end{proposition}

\begin{proof}
Fix compact subsets
\[
K\Subset K'\Subset M\setminus Z.
\]
By Proposition~\ref{prop:Section-global-C0}, the functions \(u_t\)
are uniformly bounded. By Proposition~\ref{prop:Section-local-C2},
\[
\omega\leq\eta_t\leq C_{K'}\omega
\qquad\text{on }K'.
\]
Thus,  let \(\lambda_1,\ldots,\lambda_n\) be the eigenvalues of
\(\eta_t\) with respect to \(\omega\), then
\begin{equation}\label{eq:Section-eigenvalue-two-sided}
1\leq\lambda_i\leq C_{K'},
\qquad 1\leq i\leq n,
\end{equation}
uniformly for \(0<t\leq t_0\).

We first verify uniform ellipticity and concavity of the equation.
In local holomorphic coordinates, let \(A=(A_{i\bar j})>0\) be a
positive definite Hermitian matrix and write
\(W=(\omega_{i\bar j})\). Define
\[
\mathcal F_t(x,A)
:=
-\operatorname{tr}(A^{-1}W)
-c_t\frac{\det W}{\det A}.
\]
Then the twisted \(J\)-equation
\[
\eta_t^n
=
n\eta_t^{n-1}\wedge\omega+c_t\omega^n
\]
is exactly
\begin{equation}\label{eq:Section-F-equation}
\mathcal F_t(x,\eta_t)=-1.
\end{equation}

At a fixed point, choose \(\omega\)-unitary coordinates in which
\(A=\operatorname{diag}(\lambda_1,\ldots,\lambda_n)\), and let
\[
q:=\frac{c_t}{\lambda_1\cdots\lambda_n}.
\]
The first variation of \(\mathcal F_t\) is
\[
D_A\mathcal F_t(H)
=
\operatorname{tr}(A^{-1}HA^{-1}W)
+
q\,\operatorname{tr}(A^{-1}H).
\]
In particular,
\begin{equation}\label{eq:Section-F-linearization}
\mathcal F_t^{\,i\bar j}
=
\delta_{ij}
\left(
\frac1{\lambda_i^2}+\frac{q}{\lambda_i}
\right)
\end{equation}
at such a point. Since \(c_t\) is uniformly bounded and
\eqref{eq:Section-eigenvalue-two-sided} holds, there exist positive
constants \(\kappa_K,\Lambda_K\), independent of \(t\), such that
\begin{equation}\label{eq:Section-uniform-ellipticity}
\kappa_K\sum_i|\xi_i|^2
\leq
\sum_{i,j}\mathcal F_t^{\,i\bar j}\xi_i\bar\xi_j
\leq
\Lambda_K\sum_i|\xi_i|^2
\end{equation}
on \(K'\). Thus the equation is uniformly elliptic there, with
ellipticity constants independent of \(t\).

We next verify concavity. For a Hermitian matrix \(H\), differentiating
twice in the direction \(H\) gives
\begin{align}
D_A^2\mathcal F_t(H,H)
={}&
-2\operatorname{tr}
\bigl(A^{-1}HA^{-1}HA^{-1}W\bigr) \notag\\
&-
c_t\frac{\det W}{\det A}
\left\{
\bigl[\operatorname{tr}(A^{-1}H)\bigr]^2
+
\operatorname{tr}(A^{-1}HA^{-1}H)
\right\}.
\label{eq:Section-F-concavity}
\end{align}
Every term on the right-hand side is nonpositive. Indeed, setting
\(P=A^{-1/2}HA^{-1/2}\) and
\(Q=A^{-1/2}WA^{-1/2}>0\), we have
\[
\operatorname{tr}
\bigl(A^{-1}HA^{-1}HA^{-1}W\bigr)
=
\operatorname{tr}(P^2Q)\geq0,
\]
while
\(\operatorname{tr}(A^{-1}HA^{-1}H)=\operatorname{tr}(P^2)\geq0\).
Consequently,
\[
D_A^2\mathcal F_t(H,H)\leq0,
\]
so \(\mathcal F_t(x,\cdot)\) is concave on the positive definite
cone.

We now apply the complex Evans--Krylov estimate. In local coordinates
the equation takes the form
\[
\mathcal F_t
\bigl(
x,\chi(x)+t\omega(x)
+(u_t)_{i\bar j}(x)
\bigr)
=-1.
\]
The background tensors \(\chi+t\omega\) and \(\omega\) have
\(C^m\)-norms bounded independently of \(t\), for every fixed \(m\).
Together with the uniform \(C^0\)-estimate,
\eqref{eq:Section-eigenvalue-two-sided},
\eqref{eq:Section-uniform-ellipticity}, and the concavity proved
above, the interior complex Evans--Krylov estimate
(see, for example, \cite{Szekelyhidi2018}) yields some
\(\alpha\in(0,1)\), independent of \(t\), such that
\begin{equation}\label{eq:Section-C2alpha}
\|u_t\|_{C^{2,\alpha}(K)}
\leq C_K.
\end{equation}
Here \(C^{2,\alpha}\) denotes the usual real H\"older norm in local
coordinates.

It remains to bootstrap the estimate. Differentiating
\eqref{eq:Section-F-equation} in a coordinate direction \(z^\ell\)
gives
\[
\mathcal F_t^{\,i\bar j}
\partial_i\partial_{\bar j}(\partial_\ell u_t)
=
-\partial_\ell^x\mathcal F_t
-\mathcal F_t^{\,i\bar j}
\partial_\ell(\chi_{i\bar j}+t\omega_{i\bar j}),
\]
 where $\partial_\ell^x\mathcal F_t$ differentiates only the
$x$-dependent coefficients of the operator. By \eqref{eq:Section-C2alpha}, the coefficients
\(\mathcal F_t^{\,i\bar j}\) are uniformly \(C^\alpha\) on compact
subsets of \(M\setminus Z\), and
\eqref{eq:Section-uniform-ellipticity} gives uniform ellipticity.
The right-hand side is also uniformly \(C^\alpha\), since
\(\chi,\omega\) are fixed smooth forms and \(c_t\) is constant on
\(M\). The interior Schauder estimate therefore implies, after shrinking from
\(K'\) to \(K\),
\[
\|u_t\|_{C^{3,\alpha}(K)}\leq C_K.
\]
Differentiating the equation repeatedly and applying the same
Schauder estimate inductively implies, for every integer \(k\geq0\),
\[
\|u_t\|_{C^k(K)}\leq C_{K,k}.
\] This completes the proof.
\end{proof}

\subsection{Convergence to a bounded potential solution}

We now pass to the limit as \(t\to0\). The global \(C^0\) estimate
implies a bounded limiting potential, while the local higher-order
estimates ensure that the limit is smooth on \(M\setminus Z\).

\begin{theorem}\label{thm:potential-solution-Section}
Under normalization~\eqref{normalizationcondition} and smooth boundary cone condition~\eqref{semisubsolutioncondition},  there exist a sequence
\(t_j\to0^+\) and a function
\[
u_\infty\in L^\infty(M)\cap C^\infty(M\setminus Z),\quad \sup_{M}u_\infty=0,
\]
such that
\begin{equation}\label{eq:Section-convergence}
u_{t_j}\longrightarrow u_\infty
\quad\text{in }L^1(M),
\qquad
u_{t_j}\longrightarrow u_\infty
\quad\text{in }C^\infty_{\mathrm{loc}}(M\setminus Z).
\end{equation}
Moreover,
\(\chi+\sqrt{-1}\partial\bar\partial u_\infty
\geq \omega\)
in the sense of currents, and
\begin{equation}\label{eq:Section-limit-equation}
(\chi+\sqrt{-1}\partial\bar\partial u_\infty
)^n=n(\chi+\sqrt{-1}\partial\bar\partial u_\infty
)^{n-1}\wedge\omega
\end{equation}
on \(M\) in the Bedford--Taylor sense. In particular, the equation
holds smoothly on \(M\setminus Z\).
\end{theorem}

\begin{proof}Since
\(\chi+\sqrt{-1}\partial\bar\partial u_t
=\eta_t-t\omega>(1-t)\omega,
\)
and \(\sup_Mu_t=0\), the standard compactness theorem for normalized
\(\chi\)-plurisubharmonic functions gives, after passing to a
subsequence, a function
\[
u_\infty\in\operatorname{PSH}(M,\chi),
\qquad
\sup_Mu_\infty=0,
\]
such that
\(
u_{t_j}\longrightarrow u_\infty\) in \(L^1(M)\).

By Proposition~\ref{prop:Section-global-C0},
\(\|u_t\|_{L^\infty(M)}\leq C\). Hence \(u_\infty\in L^\infty(M)\). 
Also, passing to the weak limit, the inequality
\(\chi+\sqrt{-1}\partial\bar\partial u_{t_j}>
(1-t_j)\omega\) immediately implies
\(
\chi+\sqrt{-1}\partial\bar{\partial}u_\infty\geq\omega
\)
in the sense of currents.
By Proposition~\ref{prop:Section-higher-order} and a diagonal
Arzel\`a--Ascoli argument, after passing to a further subsequence,
\(u_{t_j}\) converges in \(C^\infty_{\mathrm{loc}}(M\setminus Z)\).
The smooth local limit agrees with the \(L^1\)-limit \(u_\infty\).
Since \(c_{t_j}\to0\), passing to the limit in
\eqref{eq:Section-approx-equation} gives
\[
(\chi+\sqrt{-1}\partial\bar{\partial}u_\infty)^n=n(\chi+\sqrt{-1}\partial\bar{\partial}u_\infty)^{n-1}\wedge\omega
\qquad\text{smoothly on }M\setminus Z.
\]

It remains only to interpret the equality globally. Since
\(u_\infty\in L^\infty(M)\), the Bedford--Taylor products
\((\chi+\sqrt{-1}\partial\bar{\partial}u_\infty)^k\), \(1\leq k\leq n\), are well defined; see
\cite{BedfordTaylor1982}. Moreover, these products do not charge
pluripolar sets. The analytic set \(Z\) is pluripolar, hence both
positive measures
\[
(\chi+\sqrt{-1}\partial\bar{\partial}u_\infty)^n,\qquad
n(\chi+\sqrt{-1}\partial\bar{\partial}u_\infty)^{n-1}\wedge\omega
\]
give zero mass to \(Z\). Since they agree on \(M\setminus Z\), they
agree on all of \(M\). This proves
\eqref{eq:Section-limit-equation}.
\end{proof}

\begin{remark}\label{rem:uniqueness-and-full-convergence}
Murakami~\cite[Theorem~1.7]{Murakami2026} proved that, under the
boundary approximation hypothesis, there exists a unique normalized
weak solution of the boundary generalized Monge--Amp\`ere equation,
where both the equation and the corresponding closed cone inequalities
are understood in terms of non-pluripolar products. In the case of the
\(J\)-equation, this gives a unique normalized \(\chi\)-psh function
\(u_T\) satisfying
\(
\left\langle T^n\right\rangle
=
n\left\langle
T^{n-1}\wedge\omega
\right\rangle,
\)
together with
\(
\left\langle T^p\right\rangle
-
p\left\langle
T^{p-1}\wedge\omega
\right\rangle
\geq0\)
for \(1\leq p\leq n-1,
\)
where
\(
T
=
\chi+\sqrt{-1}\partial\bar\partial u_T
\) is a $(1,1)$-current.

We claim that
\(
u_\infty=u_T.
\)
Thus our theorem can be viewed as a regularity upgrade of Murakami's
weak solution under the smooth boundary cone condition: the latter has a bounded
potential, smooth on \(M\setminus Z\) and satisfies the \(J\)-equation globally in the
Bedford--Taylor sense.
Indeed, the boundary approximation hypothesis of
\cite[Assumption~1.4]{Murakami2026} is satisfied in our setting.
Moreover, \(
T_\infty
:=
\chi+\sqrt{-1}\partial\bar\partial u_\infty
\) is a smooth $(1,1)$-form on \(M\setminus Z\), satisfying
\(
T_\infty^n
=
nT_\infty^{n-1}\wedge\omega\) on $M$ in the Bedford--Taylor sense, and
\(F_p(T_\infty)>0\) on \(M\setminus Z\) for
\(1\leq p\leq n-1.
\)
Since \(Z\) is analytic and hence pluripolar, the non-pluripolar
products do not charge \(Z\). Therefore the above inequalities extend
across \(Z\), and hence
\(
\left\langle T_\infty^p\right\rangle
-
p\left\langle
T_\infty^{p-1}\wedge\omega
\right\rangle
\geq0
\) on $M$ for
\(1\leq p\leq n-1.
\)
Thus \(u_\infty\) satisfies all the hypotheses of
\cite[Theorem~1.7]{Murakami2026}. By uniqueness~\cite[Theorem~1.7]{Murakami2026},
\(
u_\infty=u_T.
\)
\end{remark}

Finally, we prove the uniform \(C^0\)-estimate for the \(J\)-flow under the smooth boundary cone condition. The proof is an application of Proposition~\ref{prop:Section-global-C0}.
\begin{corollary}[Uniform \(C^0\)-estimate for the \(J\)-flow]
\label{cor:J-flow-global-C0}
Under the assumptions of
Theorem~\ref{thm:regularity of the weak solution}, let
\(\varphi=\varphi(x,s)\) be the solution of the normalized \(J\)-flow
\begin{equation}\label{eq:J-flow}
\frac{\partial\varphi}{\partial s}
=
\frac{1}{n}
\left(
1-\operatorname{tr}_{\chi_\varphi}\omega
\right),
\qquad
\chi_\varphi
:=
\chi+\sqrt{-1}\partial\bar\partial\varphi>0,
\qquad
\varphi(\cdot,0)=\varphi_0,
\end{equation}
where \(\varphi_0\in C^\infty(M)\) and
\(\chi_{\varphi_0}>0\). Then there exists a constant \(C>0\),
depending only on the background data and \(\varphi_0\), such that
\[
\sup_{s\geq0}\|\varphi(\cdot,s)\|_{L^\infty(M)}\leq C.
\]
\end{corollary}

\begin{proof}
For \(0<t\leq t_0= \frac{1}{4}\), let \(u_t\) solve the approximating twisted
\(J\)-equation~\eqref{eq:Section-approx-equation} and set
\(
\eta_t
=
\chi+t\omega+\sqrt{-1}\partial\bar\partial u_t.
\) 
Let \(\lambda_1,\ldots,\lambda_n\) be the eigenvalues of \(\eta_t\)
with respect to \(\omega\), then
\begin{equation}\label{eq:J-flow-eigenvalue-approximation}
1
=
\sum_{i=1}^n\frac{1}{\lambda_i}
+
\frac{c_t}{\lambda_1\cdots\lambda_n},
\qquad
\lambda_i>1,
\qquad
0<c_t\leq Ct.
\end{equation}
In particular,
\[
\chi_{u_t}
:=
\chi+\sqrt{-1}\partial\bar\partial u_t
=
\eta_t-t\omega>(1-t)\omega>0.
\]
Since the eigenvalues of \(\chi_{u_t}\) with respect to \(\omega\)
are \(\lambda_i-t\), we obtain
\begin{align}
\left|
1-\operatorname{tr}_{\chi_{u_t}}\omega
\right|
&=
\left|
\frac{c_t}{\lambda_1\cdots\lambda_n}
-
\sum_{i=1}^n
\frac{t}{\lambda_i(\lambda_i-t)}
\right| \notag\\
&\leq
c_t+\frac{t}{1-t}\sum_{i=1}^n\frac{1}{\lambda_i}
\leq C_0t
\label{eq:J-flow-approximate-stationary}
\end{align}
for some uniform constant \(C_0>0\).

By Proposition~\ref{prop:Section-global-C0}, there exists \(C_1>0\),
independent of \(t\), such that
\(\|u_t\|_{L^\infty(M)}\leq C_1\). Set
\[
A:=C_1+\|\varphi_0\|_{L^\infty(M)}.
\]
Then \(u_t-A\leq\varphi_0\leq u_t+A\) on \(M\). For fixed \(t\), define
\[
\underline\varphi_t(x,s)
:=
u_t(x)-A-\frac{C_0}{n}ts,
\qquad
\overline\varphi_t(x,s)
:=
u_t(x)+A+\frac{C_0}{n}ts.
\]
Then by~\eqref{eq:J-flow-approximate-stationary},
\[
\begin{aligned}
\frac{\partial\underline\varphi_t}{\partial s}
&=
-\frac{C_0}{n}t
\leq
\frac{1}{n}
\left(
1-\operatorname{tr}_{\chi_{\underline\varphi_t}}\omega
\right),\qquad 
\frac{\partial\overline\varphi_t}{\partial s}
&=
\frac{C_0}{n}t
\geq
\frac{1}{n}
\left(
1-\operatorname{tr}_{\chi_{\overline\varphi_t}}\omega
\right).
\end{aligned}
\]
 Hence the parabolic
comparison principle implies
\begin{equation}\label{eq:J-flow-comparison-elliptic}
u_t-A-\frac{C_0}{n}ts
\leq
\varphi(\cdot,s)
\leq
u_t+A+\frac{C_0}{n}ts
\qquad\text{on }M.
\end{equation}

For each \(s\geq0\), choose
\(
t=t(s):=\min\left\{t_0,\frac{1}{1+s}\right\}.
\)
Then \(t(s)s\leq1\). Applying
\eqref{eq:J-flow-comparison-elliptic} with \(t=t(s)\) and using the
uniform bound for \(u_t\), we conclude that
\[
\|\varphi(\cdot,s)\|_{L^\infty(M)}
\leq
C_1+A+\frac{C_0}{n}
\]
for every \(s\geq0\). This proves the corollary.
\end{proof}

\section*{Acknowledgements}

The authors would like to thank Xuan Li and Dekai Zhang for many helpful discussions.

\end{document}